\documentclass[10pt,a4paper]{article}
\usepackage[utf8]{inputenc}
\usepackage[T1]{fontenc}
\usepackage{geometry}
\usepackage{amsthm}
\usepackage[french,english]{babel}
\usepackage{amssymb}
\usepackage{lmodern}
\usepackage{amsmath,amsfonts}
\usepackage{mathrsfs} 
\usepackage{dsfont}
\usepackage{stmaryrd}

\usepackage{graphicx}
\usepackage{fullpage}
\usepackage[nottoc,notlot,notlof]{tocbibind}
\usepackage[colorlinks=true,urlcolor=black,linkcolor=black,citecolor=black]{hyperref}
\usepackage{hyperref}
\numberwithin{equation}{section}
\usepackage{tikz}
\usetikzlibrary{matrix}
\usepackage{centernot}

\begin{document}
	\renewcommand{\contentsname}{Table of Contents}
	\renewcommand{\abstractname}{Abstract}
	
	\title{A Viscosity Method for the Min-Max Construction of Closed Geodesics}
	\author{Alexis Michelat\footnote{Department of Mathematics, ETH Zentrum, CH-8093 Zürich, Switzerland.}\; and Tristan Rivière$^{*}$\setcounter{footnote}{0}}
	\date{\today} 
	
	\maketitle
	
	\begin{center}
		\textit{Dedicated to Jean-Michel Coron at the occasion of his $60^{th}$ anniversary}
	\end{center}
	
	\vspace{1.5em}
	
	\begin{abstract}
		We present a viscosity approach to the min-max construction of closed geodesics on compact Riemannian manifolds of arbitrary dimension. We also construct counter-examples in dimension $1$ and $2$ to the $\varepsilon$-regularity in the convergence procedure. Furthermore, we prove the lower semi-continuity of the index of our sequence of critical points converging towards a closed non-trivial geodesic.
	\end{abstract}
	
	\tableofcontents
	
	\makeatletter
	\tracinggroups=1
	\theoremstyle{plain}
	\newtheorem{theorem}{Theorem}[section]
	\newtheorem{theoremdef}{Théorème-Définition}[section]
	\newtheorem{lemme}[theorem]{Lemma}
	\newtheorem{propdef}[theorem]{Définition-Proposition}
	\newtheorem{prop}[theorem]{Proposition}
	\newtheorem{cor}[theorem]{Corollary}
	\theoremstyle{definition}
	\newtheorem{defi}[theorem]{Definition}
	\newtheorem{rem}[theorem]{Remark}
	\newtheorem{exemple}[theorem]{Example}
	\newcommand{\N}{\ensuremath{\mathbb{N}}}
	\renewcommand\hat[1]{%
		\savestack{\tmpbox}{\stretchto{%
				\scaleto{%
					\scalerel*[\widthof{\ensuremath{#1}}]{\kern-.6pt\bigwedge\kern-.6pt}%
					{\rule[-\textheight/2]{1ex}{\textheight}}
				}{\textheight}%
			}{0.5ex}}%
		\stackon[1pt]{#1}{\tmpbox}%
	}
	\parskip 1ex
	\renewcommand{\H}{\ensuremath{\mathbb{H}}}
	\newcommand{\vc}[3]{\overset{#2}{\underset{#3}{#1}}}
	\newcommand{\conv}[1][]{\ensuremath{\xrightarrow[n\rightarrow\infty]{#1}}}
	\newcommand{\cn}{\ensuremath{\mathrm{cn}}}
	\newcommand{\sn}{\ensuremath{\mathrm{sn}}}
	\newcommand{\dn}{\ensuremath{\mathrm{dn}}}
	\newcommand{\A}{\ensuremath{\mathscr{A}}}
	\newcommand{\D}{\ensuremath{\nabla}}
	\renewcommand{\N}{\ensuremath{\mathbb{N}}}
	\newcommand{\Z}{\ensuremath{\mathbb{Z}}}
	\newcommand{\K}{\ensuremath{\mathscr{K}}}
	\newcommand{\R}{\ensuremath{\mathbb{R}}}
	\newcommand{\Q}{\ensuremath{\mathbb{Q}}}
	\newcommand{\C}{\ensuremath{\mathbb{C}}}
	\newcommand{\I}{\ensuremath{\mathbb{I}}}
	\renewcommand{\P}{\ensuremath{\mathscr{P}}}
	\newcommand{\Res}{\ensuremath{\mathrm{Res}}}
	\newcommand{\res}{\ensuremath{\mathcal{R}}}
	\newcommand{\fond}{\ensuremath{\mathbb{I}}}
	\newcommand{\lp}[2]{\ensuremath{\mathrm{L}^{#1}(#2)}}
	\newcommand{\W}[2]{\ensuremath{\mathrm{W}^{#1}(#2)}}
	\newcommand{\WW}{\ensuremath{\mathrm{W}^{2,2}(S^1,M)}}
	\newcommand{\ww}{\ensuremath{\mathrm{W}^{2,2}(S^1,M)}}
	\newcommand{\wt}{\ensuremath{\mathrm{W}_{u}^{2,2}(S^1,TM)}}
	\newcommand{\wtn}{\ensuremath{\mathrm{W}_{u_n}^{2,2}(S^1,TM)}}
	\newcommand{\im}{\ensuremath{\mathrm{W}_{\iota}^{2,2}(S^1,M)}}
	\renewcommand{\wp}[1]{\enuremath{\dot{W}^{#1}(\R^n)}}
	\newcommand{\nw}[2]{\ensuremath{\left\Vert #1\right\Vert_{\mathrm{W}^{#2}}}}
	\newcommand{\nwp}[3]{\ensuremath{\left\Vert #1\right\Vert_{\mathrm{W}^{#2}(#3)}}}
	\newcommand{\np}[3][S^1]{\ensuremath{\left\Vert #2\right\Vert_{\mathrm{L}^{#3}(#1)}}}
	\newcommand{\npl}[2]{\ensuremath{\left\Vert #1\right\Vert_{\mathrm{L}^{#2}([0,L_n])}}}
	\newcommand{\npr}[2][2]{\ensuremath{\left\Vert #2\right\Vert_{\mathrm{L}^{#1}(C_r)}}}
	\newcommand{\nprj}[2][2]{\ensuremath{\left\Vert #2\right\Vert_{\mathrm{L}^{#1}(C_{2^jR})}}}
	\newcommand{\NP}[3]{\ensuremath{\left\Vert #1\right\Vert_{\mathrm{L}^{#2}(#3)}}}
	\newcommand{\so}[2]{\ensuremath{\mathrm{H}^{#2,#1}(\R^n)}}
	\newcommand{\ns}[2][s]{\ensuremath{\left\Vert #2 \right\Vert_{\dot{H}^{#1}(\R^n)}}}
	\newcommand{\jap}[1]{\ensuremath{(1+|\xi|^2)^{#1}}}
	\newcommand{\sch}{\ensuremath{\mathscr{S}}}
	\newcommand{\h}{\ensuremath{\mathscr{H}}}
	\newcommand{\hh}[1][s]{\ensuremath{\dot{H}^{#1}(\R^n)}}
	\newcommand{\hhh}{\ensuremath{\dot{H^1}(\R^n)\cap\dot{H}^s(\R^n)}}
	\renewcommand{\Re}{\ensuremath{\mathrm{Re}\,}}
	\renewcommand{\Im}{\ensuremath{\mathrm{Im}\,}}
	\newcommand{\diam}{\ensuremath{\mathrm{diam}\,}}
	\newcommand{\leb}{\ensuremath{\mathscr{L}}}
	\newcommand{\Hom}{\ensuremath{\mathrm{Hom}}}
	\newcommand{\supp}{\ensuremath{\mathrm{supp}\,}}
	\renewcommand{\phi}{\ensuremath{\varphi}}
	\renewcommand{\epsilon}{\ensuremath{\varepsilon}}
	\renewcommand{\bar}{\ensuremath{\overline}}
	\newcommand{\scal}[2]{\ensuremath{\left\langle #1,#2\right\rangle}}
	\newcommand{\norm}[1]{\ensuremath{\left\Vert #1\right\Vert}}
	\newcommand{\ens}[1]{\ensuremath{\left\{ #1\right\}}}
	\newcommand{\lie}[2]{\ensuremath{\left[#1,#2\right]}}
	\newcommand{\ffg}{\ensuremath{\vec{\Phi}}}
	\newcommand{\g}{\ensuremath{g_{\ffg}}}
	\newcommand{\vg}{\ensuremath{\mathrm{vol}_{\g}}}
	\newcommand{\inj}{\ensuremath{\mathrm{inj}}}
	\renewcommand{\qedsymbol}{}
	\renewcommand{\mathring}{\ensuremath{\mathrm}}
	\renewcommand{\refname}{References}
	\renewcommand{\abstractname}{Abstract}
	\renewcommand{\proofname}{Proof}
\renewcommand*{\thefootnote}{\arabic{footnote}}
\newpage

 \addtocontents{toc}{\protect\setcounter{tocdepth}{1}}

\section{Introduction}
\subsection{General Framework}

This article intends at motivating the approach developed in \cite{maxou} in the simpler case of the construction of closed geodesics. We present first the general framework of the kind of problems this method aims at tackling. 

Suppose we want to construct a critical point of a $C^1$ function $f:X\rightarrow \R_+$, where $X$ is a complete $C^{1,1}$ Finsler manifold, which we interpret as the energy of a geometric or physical problem. A critical point $x\in X$ of $f$ is non-trivial if its energy is positive, \textit{i.e.} if $Df(x)=0$ and $f(x)>0$. If\, $\inf f(X)=0$, we cannot simply minimise $f$ to create a non-trivial critical point, so we use a so-called min-max construction. Let us denote $\P^*(X)=\P(X)\setminus \ens{\varnothing}$ the set of non-empty subsets of $X$. If $\mathscr{A}\subset \mathscr{P}^*(X)$, we define the min-max
\begin{align*}
\beta=\inf_{A\in\mathscr{A}}\sup_{x\in A}f(x).
\end{align*}
Thanks of general theorem such as "mountain pass" (see for example \cite{struwevar}), if $\beta<\infty$ and if the function $f$ satisfies the Palais-Smale condition on $X$, under suitable assumptions on $\mathscr{A}$, $\beta$ is a critical value of $f$. So if $\beta>0$, we get a non-trivial critical point of $f$. We recall that $f$ satisfies the Palais-Smale condition at $c\in \R$ if for all sequence $\ens{x_n}_{n\in\N}\subset X$, such that
\begin{align*}
f(x_n)\conv c,\quad\text{and}\quad Df(x_n)\conv 0,
\end{align*}
there exists $x\in X$ and a subsequence of $\ens{x_n}_{n\in\N}$ converging to $x$. 
In general, a lack of coerciveness can prevent the energy to verify the Palais-Smale condition. The viscosity method consists in approximating $f$ by a function satisfying the Palais-Smale condition. If $g:X\rightarrow\R_+$ is $C^1$, we set, for all $\sigma\geq 0$,
\begin{align*}
f_\sigma(x)=f(x)+\sigma^2 g(x),
\end{align*}
and define
\begin{align*}
\beta(\sigma)=\inf_{A\in\mathscr{A}}\sup_{x\in A}f_\sigma(x).
\end{align*}
If for all $\sigma>0$, $f_\sigma$ verifies the Palais-Smale condition, and $\beta(\sigma)<\infty$, then we can get a critical point $x_\sigma\in X$ such that
\begin{align*}
f_\sigma(x_\sigma)=\beta(\sigma).
\end{align*}
We can easily see that
\begin{align*}
\beta(\sigma)\underset{\sigma\rightarrow 0}{\longrightarrow} \beta(0)>0,
\end{align*}
and at this point if we can construct a sequence of positive numbers $\ens{\sigma_n}_{n\in\N}$, and a sequence $\ens{x_{\sigma_n}}_{n\in\N}$ of critical points associated to $\ens{f_{\sigma_n}}_{n\in\N}$ such that
\begin{align*}
\sigma_n^2\, g(x_{\sigma_n})\conv 0,
\end{align*}
if we manage to extract a subsequence of $\ens{x_{\sigma_n}}_{n\in\N}$ converging strongly enough to an element $x\in X$, such that
\begin{align*}
f(x_{\sigma_n})\conv f(x),\;\text{ and }\; Df(x_{\sigma_n})\conv Df(x), 
\end{align*}
then $x\in X$ will be critical point of $f$ of non-trivial energy $\beta=\beta(0)>0$. 

 One new feature of our work is the absence of $\epsilon$-regularity, as the counter-examples shows in section \ref{counterexamples}. The convergence is assured instead by the existence of a quasi-conservation law.

\subsection{Construction of Closed Geodesics}

The problem of the construction of closed geodesics on compact manifolds is an ancient problem which has stimulated great developments in the field of calculus of variations, dynamical systems (\cite{bangert}, \cite{Geodesicsannulushomeomorphisms}) and algebraic topology (\cite{bott1}, \cite{sullivan}). After the pioneering work of Hadamard (\cite{hadamard}) and Poincaré (\cite{poinca}), the first existence results on $2$-dimensional spheres equipped with arbitrary metric were obtained by Birkhoff in 1917 (\cite{birkhoff}) and in 1927 for the general case of spheres of higher dimension (we refer to \cite{coldingminicozzi1} for a modern proof). We refer to \cite{Poincare}, \cite{type} and \cite{Openproblems} for historical developments, and to \cite{bott1} for a more mathematical treatment of the subject.

Let $(M^m,g)$ a compact Riemannian manifold of class $C^\nu$ ($\nu\geq 3$). According to the notations in the beginning of the introduction, we have $X=\im$, where
\renewcommand*{\thefootnote}{\fnsymbol{footnote}}
\begin{equation}
\im=\mathrm{W}^{2,2}(S^1,M)\cap\ens{u:u(t)\in M,\;\text{and}\; \dot{u}(t)\neq 0 \;\text{for all}\; t\in S^1}\footnote{Note that this makes sense thanks of the Sobolev imbedding $\mathrm{W}^{2,2}(S^1,M)$.}
\end{equation}
\renewcommand*{\thefootnote}{\arabic{footnote}}
\setcounter{footnote}{0}
and $f=\mathfrak{L}$, the length of curve, such that for all $u\in\im$,
\begin{align*}
\mathfrak{L}(u)=\int_{S^1}\left|\dot{u}\right|d\leb^1.
\end{align*}
and $g=\kappa^2$, where
\begin{align*}
\kappa(u)=\left|\D_{\frac{\dot{u}}{|\dot{u}|}}\frac{\dot{u}}{|\dot{u}|}\right|
\end{align*}
is the geodesic curvature of a curve $u\in \im$. We then consider for all $\sigma\geq 0$, the energy $E_\sigma:\im\rightarrow \R$, defined for all $u\in\im$ by
\begin{align*}
E_\sigma(u)=\int_{S^1}(1+\sigma^2\kappa^2(u))|\dot{u}| d\leb^1.
\end{align*}
Of course, we can replace $f$ by the Dirichlet energy, which verifies the Palais-Smale condition : it is a classical way to construct a closed geodesic on compact manifolds (see for instance \cite{struwevar}). However, we are interested in the application of this method to the min-max construction of minimal surfaces, and the Dirichlet energy does not satisfy any more the Palais-Smale condition in dimension $2$. Therefore it makes sense to consider first a simpler case, to see if the method works correctly, and where are the difficulties. Indeed, there are three issues that we might encounter. 

Firstly, we need to construct an appropriate min-max method, giving a $\beta(0)>0$. Secondly, if $\ens{u_n}_{n\in\N}$ is a sequence of critical points associated to $\ens{E_{\sigma}}$ (where $\ens{\sigma_n}_{n\in\N}$ is a sequence of positive numbers converging to $0$)
\begin{align*}
\liminf_{n\rightarrow\infty}\int_{S^1}\sigma_n^2\kappa^2(u_n)d\leb^1\conv 0.
\end{align*}

Thirdly, passing to the limit in the Euler-Lagrange equation is delicate, as we loose the estimates on the second derivative. 

The first problem can easily be solved, using basic properties of the injective radius of compact manifolds. For the second one, there exist indeed counter-examples, and we use a general technique coming from an article of Michael Struwe (\cite{struweart}) to construct an "entropic" sequence of critical points, in the sense that
\begin{align*}
E_{\sigma_n}(u_n)=\beta(\sigma_n),\quad \text{and}\; \int_{S^1}\sigma^2\kappa^2(u)d\leb^1\leq \frac{1}{\log\frac{1}{\sigma_n}}.
\end{align*}
Finally, the limiting procedure depends on a quasi-conservation law of the Euler-Lagrange equation, corresponding to the general scheme of Noether theorem (see \cite{helein}).

We are almost in the position of stating our main result. We first recall that the index of a critical point $u\in\im$ of $E_\sigma$ ($\sigma\geq 0$ arbitrary) is the dimension of the maximal vector subspace of $\mathrm{W}^{2,2}_u(S^1,M)$ where the second derivative $D^2E_\sigma(u)$ is negative semi-definite.

For the definition of admissible sets and of the families of maps $\mathscr{A},\mathscr{A}_0$, we refer to section \ref{bonnesuite}.

\begin{theorem}\label{mainresult}
	Let $(M^m,g)$ a compact Riemannian manifold of class $C^\nu$, ($\nu\geq 3$). There exists an admissible set $\A$ for $\mathrm{W}^{2,2}(S^1,M)$, a sequence of positive numbers $\ens{\sigma_n}_{n\in\N}$ converging to $0$ and a sequence of critical points $\ens{u_n}_{n\in\N}$ associated to $\ens{E_{\sigma_n}}_{n\in\N}$, such that if we define
	\[
	\beta(\sigma_n)=\inf_{A_0\in\A_0}\sup_{u\in A_0}E_{\sigma_n}(u)<\infty,\quad \beta(0)=\inf_{A\in\A}\sup_{u\in A}\mathfrak{L}(u)>0,
	\]
	then
	\[
	E_{\sigma_n}(u_n)=\beta(\sigma_n), \quad \sigma_n^2\int_{S^1}\kappa^2(u_n)|\dot{u}_n|d\leb^1\leq \frac{1}{\log\frac{1}{\sigma_n}},
	\]
	and 
	\begin{align}\label{strong}
	u_n\overset{\mathrm{L}^\infty}\conv u,\quad  \text{and}\quad\dot{u}_n\overset{a.e.}{\conv } \dot{u}
	\end{align}
	where $u$ is a closed non-trivial $C^\nu$ geodesic of length $\beta(0)>0$. Furthermore, we have lower semi-continuity of the index, \text{i.e.}
	\begin{align*}
	\mathrm{Ind}(u)\leq \liminf_{n\rightarrow\infty}\mathrm{Ind}(u_n).
	\end{align*}
\end{theorem}
\begin{proof}
	The proof is the reunion of theorems \ref{bonnextraction}, \ref{géodésique}, \ref{admissiblefamily} and \ref{indextheorem}.
\end{proof}

Methods of viscosity were already successfully used in the past in various contexts: in elliptic partial differential equations (\cite{struweart}), hyperbolic partial differential equations (\cite{tartar1}, \cite{tartar2}) harmonic maps from surfaces (\cite{sacks}, \cite{tamashi}), and recently by the second author for free boundary problems (\cite{kekkon}), Yang-Mills equations (\cite{tian}). One general feature in these pieces is the $\epsilon$-regularity that one can get independently of $\sigma$. For example, in \cite{tamashi}, we consider immersions of a Riemannian surface $(M^2,g)$ into spheres $S^k$ ($k\in\N$, and $k\geq 2$), with
\begin{align*}
E_\sigma(u)=\int_{M}\left(|\D u|^2+\sigma^2|\Delta u|^2\right)d\mathrm{vol}_g,
\end{align*}
then the $\epsilon$-regularity means that there exists $\epsilon>0$, and $\delta>0$, such that for all $x\in M$, and $r>0$, there exists a constant $C=C(r,\epsilon)$ such that for all $\sigma>0$, for all critical point $u_\sigma$ of $E_\sigma$, the inequality
\begin{align*}
\int_{B_r(x)}\left(|\D u_\sigma|^2+\sigma^2|\Delta u_\sigma|^2\right)d\mathrm{vol}_g<\epsilon,
\end{align*}
implies that for all $k\in\N$, for all $0<\alpha<1$,
\begin{align*}
\Vert u_\sigma\Vert_{C^{k,\alpha}(B_{\delta r}(x))}\leq C,
\end{align*}
and this ensures that the limits of $\ens{u_\sigma}_{\sigma>0}$ are smooth, using classical results on the resolution of singularities for harmonic maps (see the references cited in \cite{tamashi}).
One new phenomena is the absence of $\epsilon$-regularity in this construction, as the following counter-examples shows (see section \ref{counterexamples} for the proof). 

\begin{prop}\label{counterintro}
	On $S^2$ equipped with its standard metric, let $\A$ the admissible set of curves given by the canonical sweep-out on $S^2$.  There exists a sequence $\ens{\sigma_n}_{n\in\N}$ of positive real numbers converging to $0$ and a sequence of critical points $\ens{u_n}_{n\in\N}$ of $\ens{E_{\sigma_n}}_{n\in\N}$, and a curve $u\in \mathrm{W}^{1,2}(S^1,M)$, such that \vspace{-0.4em}
	\begin{align*}
	E_{\sigma_n}(u_n)\conv \beta(0)=\pi, \quad \mathfrak{L}(u_n)\conv \frac{\pi}{2}
	\end{align*}
	and
	\begin{align*}
	u_n\conv[\mathrm{L}^\infty] u\quad\text{strongly},\quad \text{and}\quad u_n\xrightarrow[n\rightarrow\infty]{\mathrm{W}^{1,2}} u\quad\text{weakly},\quad \dot{u}_n\centernot{\xrightarrow[n\rightarrow\infty]{}} \dot{u} \quad \text{a.e.}
	\end{align*}
	Furthermore, there exists a negligible subset $N\subset S^1$ such that
	$\ens{\dot{u}_n(t)}_{n\in\N}$ has no limit point for all $t\in S^1\setminus N$,
	and for all open interval $I\subset S^1$,
	\begin{align*}
	\mathfrak{L}(u|I)<\liminf_{n\rightarrow\infty}\mathfrak{L}(u_n|I).
	\end{align*}
\end{prop}

Due to the absence of $\epsilon$-regularity, we had to exploit quasi-conservation law issued from "almost Noether theorem", in order to apply technics from compensated compactness for getting the strong convergence in \eqref{strong}.

Finally, we note that our approach can also be applied for the construction of non-compact manifolds admitting non-trivial closed geodesics thanks of the article of Benci and Giannoni \cite{noncompactRiemannian}.

\section{Analytic and Geometric Preliminaries}

Let $(M^m,g)$ be a compact Riemannian manifold of dimension $m$ greater than $2$, and of class $C^\nu$ (where $\nu\geq 3$). We always assume that $M$ is equipped with its Levi-Civita connection $\D$ (we refer for definitions in Riemannian geometry to \cite{paulin}, \cite{jost} and \cite{lee}, and to \cite{federer} for the definitions and notations on measures). Let us recall the definition of  Sobolev spaces used in the following. One possible construction is to embed isometrically $M$ into an euclidean space $\R^q$ ($q\in\N$) thanks of Nash isometric embedding theorem, which we can apply here because $M$ is a $C^\nu$ manifold and $\nu\geq 3$. In the following, we can suppose that $M$ is a submanifold of $\R^q$. Let us denote $S^1=\C\cap\ens{z:|z|=1}$.

\begin{defi}
The Sobolev space $\ww$ is defined as follow
\begin{align*}
\ww=\mathrm{W}^{2,2}(S^1,M)\cap \ens{u:u(t)\in M\;\text{for all}\; t\in S^1}.
\end{align*}
The space of Sobolev immersions $\im$ is
\begin{align*}
\im=\ww\cap\ens{u:\dot{u}(t)\neq 0\;\text{for all}\; t\in S^1}.
\end{align*}
Finally, the vector space of tangent vector fields along an immersion $u\in\im$ is denotes by
\begin{align*}
\wt=\mathrm{W}^{2,2}(S^1,TM)\cap\ens{v:v(t)\in T_{u(t)}M\;\text{for all}\; t\in S^1}.
\end{align*}
\end{defi}

\begin{rem}
All these conditions make sense, because of the Sobolev embedding theorem, there is a continuous injection $\mathrm{W}^{2,2}(S^1,\R^q)$ into the space $C^{1,\frac{1}{2}}(S^1,\R^q)$ of differentiable mappings with $\frac{1}{2}$-Hölder continuous derivative. 
\end{rem}

We first remark that $\im$ has a natural smooth Finsler manifold structure, modelled on the Hilbert space $\mathrm{W}^{2,2}(S^1,\R^q)$, as it is an open set of the Hilbert manifold $\mathrm{W}^{2,2}(S^1,M)$. Furthermore that for all $u\in\im$, the tangent space of $\im$ is equal to $\wt$.

\begin{defi}
The covariant derivative along an immersion $u\in\im$ induced by the Levi-Civita connexion $\D$ with be denoted $D_t$ when there is no ambiguity on the curve.
\end{defi}

We recall that an immersion $u:S^1\rightarrow M$ is said to be a \textit{geodesic} if
\begin{align*}
D_t\dot{u}=0.
\end{align*}

\section{First Variation of Energy}

For all $\sigma\geq 0$, let $E_\sigma:\im\rightarrow\R$ be given by 
\begin{equation}
E_\sigma(u)=\int_{S^1}(1+\sigma^2\kappa(u)^2)|\dot{u}|d\leb^1
\end{equation}
for all $u\in\im$, where
\[
\kappa(u)=\left| \nabla_{\frac{\dot{u}}{|\dot{u}|}}\frac{\dot{u}}{|\dot{u}|}\right|
\]
is the geodesic curvature, and $\leb^1$ is the Lebesgue measure. If $\sigma=0$, then $E_0$ coincides with the length of curve and we note 
\begin{align*}
\mathfrak{L}(u)=E_0(u)=\int_{S^1}|\dot{u}|d\leb^1
\end{align*}
for all $u\in \im$.

We will state and prove some elementary lemmas before we proceed with the derivation of the first and second variations of the energy. 

\begin{lemme} For all $\sigma\geq 0$, the energy $E_\sigma:\im\rightarrow\R$ is a $C^{\nu-1}$ function. 
\end{lemme}
\begin{proof}
Indeed, if $P:M\times\R^q\rightarrow TM$ is the orthogonal projection, then is it a $C^{\nu-1}$ function. If $F_{\sigma}:M\times\R^q\setminus\ens{0}\times \R^q$ is the mapping defined by
\[
(x,y,z)\mapsto F_\sigma(x,y,z)=(1+\sigma^2\scal{P_{x}(z)}{P_{x}(z)}_{x})\scal{y}{y}_x^{\frac{1}{2	}},
\]
then $F_\sigma$ is a $C^{\nu-1}$ function. The claim is therefore a simple consequence of Lebesgue's dominated convergence theorem, as
\[
E_\sigma(u)=\int_{S^1}F_\sigma(u,\dot{u},\ddot{u})d\leb^1.
\]
This concludes the proof of the first lemma.
\end{proof}

We will now derive formulae for the derivatives of the curvature and other geometric quantities. A variation of a curve $u\in \im$ is an differentiable map $\gamma:I\times S^1\rightarrow M$ such that $I$ is an open interval of $\R$ containing $0$, and $\gamma(0,\cdot)=u$, and for all $s\in I$, $\gamma(s,\cdot)\in\im$. The variation vector field $v\in\wt$ is defined as
\begin{align*}
v=\partial_s\gamma(s,\cdot)|_{s=0}.
\end{align*}
As a consequence, if $X=\im$, we have
\begin{align*}
DE_\sigma(u)\cdot v=\scal{DE_\sigma(u)}{v}_{T^*X,TX}=\frac{d}{ds}E_\sigma(\gamma(s,\cdot))|_{s=0}.
\end{align*}

We denote $D_t$ (resp. $D_s$) the covariant derivative along the curve $t\mapsto \gamma(\cdot, t)$ (resp. $s\mapsto \gamma(s,\cdot)$). We have the following commutation result.
\begin{lemme}\label{l1}
Under the afore mentioned hypothesis, we have
\[
D_t\partial_s \gamma(s,t)=D_s\partial_t\gamma(s,t)
\]
for all $(s,t)\in I\times S^1$,
and if $\lie{\cdot\,}{\cdot}$ is the Lie bracket, then
\begin{align*}
[\partial_s\gamma,\partial_t\gamma]=0.
\end{align*}
\end{lemme}
\begin{proof}
Let $\gamma=(\gamma_1,\cdots,\gamma_d)$ be the local expression of $\gamma$ in a local coordinates system. We have
\[
\partial_s\gamma=\sum_{k=1}^d \partial_s\gamma_k\partial_k,\quad\partial_t\gamma=\sum_{k=1}^d\partial_t\gamma_k\partial_k
\]
Thanks of the defining properties of a connexion, we have
\begin{align*}
D_t\partial_s\gamma=\sum_{i,j,k=1}^d\left(\partial^2_{ts}\,\gamma_k+\partial_s\gamma_i\,\partial_t\gamma_j\,\Gamma_{ij}^k\right)\partial_k\\
D_s\partial_t\gamma=\sum_{i,j,k=1}^d\left(\partial^2_{st}\,\gamma_k+\partial_t\gamma_i\,\partial_s\gamma_j\,\Gamma_{ij}^k\right)\partial_k
\end{align*}
as the connection is symmetric, \textit{i.e.} $\Gamma_{ij}^k=\Gamma_{ji}^k$, if suffices now to exchange the index $i$ and $j$ of one of the two preceding lines.
The second result is a consequence of the absence of torsion of Levi-Civita connection. As
\[
D_t\partial_s \gamma(s,t)=\nabla_{\partial_t\gamma(s,t)}\partial_s\gamma(s,t), \quad D_s\partial_t \gamma(s,t)=\nabla_{\partial_s\gamma(s,t)}\partial_t\gamma(s,t).
\]
we deduce that 
\begin{align*}
[\partial_s\gamma(s,t),\partial_t\gamma(s,t)]=D_t\partial_s \gamma(s,t)&=\nabla_{\partial_t\gamma(s,t)}\partial_s\gamma(s,t)-\nabla_{\partial_s\gamma(s,t)}\partial_t\gamma(s,t)\\
&=D_t\partial_s \gamma(s,t)-D_s\partial_t \gamma(s,t)=0,
\end{align*}
which completes the proof of the second lemma.
\end{proof}

We will denote in the following, if $(x,v)\in TM$,
\[
|v|=\sqrt{g_x(v,v)}=\sqrt{\langle v, v\rangle_x}.
\]
We now aim at calculating the first variation of the curvature.
\[
\kappa(s,t)=\left|D_t\left(\dfrac{\partial_t\gamma(s,t)}{|\partial_t\gamma(s,t)|}\right)\dfrac{1}{|\partial_t\gamma(s,t)|}\right|
\]
where $D_t$ is the covariant derivative along the curve $t\mapsto \gamma(\cdot,t)$. As $\gamma$ is extensive for $s$ close enough to $0$, we have
\[
\kappa(s,t)=\left|\nabla_{\frac{\partial_t\gamma(s,t)}{|\partial_t\gamma(s,t)|}}\frac{\partial_t\gamma(s,t)}{|\partial_t\gamma(s,t)|}\right|.
\]
To simplify notations, let us write $\gamma_t=\partial_t\gamma$, $\gamma_s=\partial_s\gamma$, et $\bar{\gamma_t}=\dfrac{\partial_t\gamma}{|\partial_t\gamma|}$.
\begin{prop}\label{commute}
Under the preceding hypothesis, we have the following identities
\begin{enumerate}
\item $\partial_s|\gamma_t|=\scal{\nabla_{\gamma_t}\gamma_s}{\bar{\gamma_t}}=\alpha |\gamma_t|$, where $\alpha=\scal{\D_{\bar{\gamma}_t}\gamma_s}{\gamma_t}$,
\item $\lie{\gamma_s}{\bar{\gamma_t}}=-\scal{\nabla_{\bar{\gamma_t}}\gamma_s}{\bar{\gamma_t}}\bar{\gamma_t}=-\alpha \bar{\gamma_t}$,
\item $\partial_s\kappa^2=2\scal{\nabla_{\bar{\gamma_t}}^2\gamma_s}{\nabla_{\bar{\gamma_t}}\bar{\gamma_t}}-4\alpha\kappa^2+2\scal{R(\gamma_s,\bar{\gamma_t})\bar{\gamma_t}}{\nabla_{\bar{\gamma_t}}\bar{\gamma_t}}$
 ($\alpha=\scal{\D_{\bar{\gamma_t}}\gamma_s}{\bar{\gamma_t}}$).
\end{enumerate}
\end{prop}
\begin{proof}
\begin{enumerate}
\item 
We have
\begin{align*}
\partial_s|\gamma_t|=\frac{1}{|\gamma_t|}\scal{D_s\gamma_t}{\gamma_t}
=\scal{D_t\gamma_s}{\bar{\gamma}_t}
=\scal{\nabla_{\gamma_t}\gamma_s}{\bar{\gamma}_t}.
\end{align*}
\item 	Indeed, thanks of lemma \ref{l1}, we have
\begin{align*}
	[\gamma_s,\bar{\gamma}_t]&=\nabla_{\gamma_s}\bar{\gamma}_t-\nabla_{\bar{\gamma}_t}\gamma_s=\nabla_{\partial_s\gamma}\frac{\partial_t\gamma}{|\partial_t\gamma|}-\nabla_{\frac{\partial_t\gamma}{|\partial_t\gamma|}}\partial_s		\gamma
	=D_s\frac{\partial_t\gamma}{|\partial_t\gamma|}-\frac{1}{|\partial_t\gamma|}			\nabla_{\partial_t\gamma}\partial_s\gamma\\
	&=\frac{D_s\partial_t\gamma}{|\partial_t\gamma|}-\frac{\partial_s|\partial_t			\gamma|}{|\partial_t\gamma|^2}\partial_t\gamma-\frac{D_t\partial_s\gamma}{|				\partial_t\gamma|}
	=-\scal{\nabla_{\bar{\gamma}_t}\gamma_s}{\bar{\gamma}_t}\bar{\gamma}_t
\end{align*}
\item Finally,
\begin{align*}
\partial_s\kappa^2=2
\scal{D_s\nabla_{\bar{\gamma}_t}\bar{\gamma}_t}{\nabla_{\bar{\gamma}_t}\bar{\gamma}_t}=2\scal{\nabla_{\gamma_s}\nabla_{\bar{\gamma}_t}\bar{\gamma}_t}{\nabla_{\bar{\gamma}_t}\bar{\gamma}_t}
\end{align*}
and $\nabla_{\gamma_s}\bar{\gamma}_t=\nabla_{\bar{\gamma}_t}\gamma_s-\alpha \bar{\gamma}_t$, therefore
\begin{align*}
\partial_s\kappa^2&=2\scal{\nabla_{\gamma_s}\nabla_{\bar{\gamma}_t}\bar{\gamma}_t-\nabla_{\bar{\gamma}_t}\nabla_{\gamma_s}\bar{\gamma}_t-\nabla_{[\gamma_s,\bar{\gamma}_t]}\bar{\gamma}_t}{\nabla_{\bar{\gamma}_t}\bar{\gamma}_t}+2\scal{\nabla_{\bar{\gamma}_t}\nabla_{\gamma_s}\bar{\gamma}_t}{\nabla_{\bar{\gamma}_t}\bar{\gamma}_t}\\
&+2\scal{\nabla_{[\gamma_s,\bar{\gamma}_t]}\bar{\gamma}_t}{\nabla_{\bar{\gamma}_t}\bar{\gamma}_t}\\
&=2\scal{R(\gamma_s,\bar{\gamma}_t)\bar{\gamma}_t}{\nabla_{\bar{\gamma}_t}\bar{\gamma}_t}+2\scal{\nabla_{\bar{\gamma}_t}(\nabla_{\bar{\gamma}_t}\gamma_s-\alpha \bar{\gamma}_t)}{\nabla_{\bar{\gamma}_t}\bar{\gamma}_t}+2\scal{\nabla_{-\alpha\bar{\gamma}_t}\bar{\gamma}_t}{\nabla_{\bar{\gamma}_t}\bar{\gamma}_t}\\
&=2\scal{R(\gamma_s,\bar{\gamma}_t)\bar{\gamma}_t}{\nabla_{\bar{\gamma}_t}\bar{\gamma}_t}-4\alpha\kappa^2+2\scal{\nabla_{\bar{\gamma}_t}\nabla_{\bar{\gamma}_t}\gamma_s}{\nabla_{\bar{\gamma}_t}\bar{\gamma}_t}-2\scal{dg(\bar{\gamma}_t)\bar{\gamma}_t}{\nabla_{\bar{\gamma}_t}\bar{\gamma}_t}\\
&=2\scal{R(\gamma_s,\bar{\gamma}_t)\bar{\gamma}_t}{\nabla_{\bar{\gamma}_t}\bar{\gamma}_t}-4\alpha\kappa^2+2\scal{\nabla_{\bar{\gamma}_t}\nabla_{\bar{\gamma}_t}\gamma_s}{\nabla_{\bar{\gamma}_t}\bar{\gamma}_t}
\end{align*}
as $|\bar{\gamma}_t|=1$, we deduce that $0=d\scal{\bar{\gamma}_t}{\bar{\gamma}_t}\cdot \bar{\gamma}_t=2\scal{\nabla_{\bar{\gamma}_t}\bar{\gamma}_t}{\bar{\gamma}_t}$.
\end{enumerate}
This calculation ends the proof of the proposition.
\end{proof}

\begin{prop}
If $u\in \im$, $L=\mathfrak{L}(u)$, and $v\in \wt$, then 
\begin{align}\label{firstvar}
DE_\sigma(u)\cdot v=\int_{0}^{L}\scal{D_tv}{\dot{u}}d\leb^1+\sigma^2\int_{0}^L2 \scal{D_t^2v}{D_t\dot{u}}-3\scal{D_tv}{\dot{u}}\kappa^2(u)+2\scal{R(v,\dot{u})\dot{u}}{D_t\dot{u}}d\leb^1.
\end{align}
if $R$ is the Riemannian curvature tensor on $(M^m,g)$.
\end{prop}
\begin{proof}
Thanks of the preceding lemmas, if $\gamma$ is a variation of $u$ such that $\partial_s\gamma|_{s=0}=v$, then we have
\begin{align*}
\dfrac{d}{ds} E_{\sigma}(\gamma(s,\cdot))&=\int_{S^1}\sigma^2(\partial_s\kappa^2(s,t))|\partial_t\gamma(s,t)| dt
+\int_{S^1}(1+\sigma^2\kappa^2(s,t))\partial_s{|\partial_t\gamma(s,t)|}dt\\
&=\sigma^2\int_{0}^L \scal{\nabla_{\bar{\gamma}_t}\nabla_{{\bar{\gamma}_t}}\gamma_s}{\nabla_{\bar{\gamma}_t}{\bar{\gamma}_t}}
-4\alpha\kappa^2+2\scal{R(\gamma_s,\gamma_t)\gamma_t}{D_t\gamma_t}d\leb^1
+\int_{0}^L(1+\sigma^2\kappa^2)\alpha d\leb^1\\
&=\int_{0}^{L}\scal{D_tv}{\dot{u}}+\sigma^2\int_{0}^{L}2\scal{D_t^2v}{D_t\dot{u}}d\leb^1-3\scal{D_tv}{\dot{u}}\kappa^2(u)+2\scal{R(v,\dot{u})\dot{u}}{D_t\dot{u}}d\leb^1.
\end{align*}
So we have the desired result.
\end{proof}
If $u$ is a critical point of $E_\sigma$ of at least class $C^3$, then
\begin{align*}
\frac{d}{ds}E_s(\gamma(s,\cdot))&=\int_{0}^L-2\sigma^2\scal{\nabla_{\bar{\gamma}_t}\gamma_s}{\nabla_{\bar{\gamma}_t}^2{\bar{\gamma}_t}}
+\scal{\nabla_{\bar{\gamma}_t}\gamma_s}{(1-3\sigma^2\kappa^2){\bar{\gamma}_t}}
+2\sigma^2\scal{U_s}{R(\nabla_{\bar{\gamma}_t}{\bar{\gamma}_t},{\bar{\gamma}_t}){\bar{\gamma}_t}}d\tau\\
&=\int_{0}^L\scal{2\sigma^2\nabla_{\bar{\gamma}_t}^3{\bar{\gamma}_t}
+\nabla_{\bar{\gamma}_t}((3\sigma^2\kappa^2-1)\bar{\gamma}_t)
+2\sigma^2R(\nabla_{\bar{\gamma}_t},{\bar{\gamma}_t}){\bar{\gamma}_t}}{\gamma_s}d\tau
\end{align*}
as $\scal{R(\gamma_s,\bar{\gamma}_t)\bar{\gamma}_t}{\nabla_{\bar{\gamma}_t}\bar{\gamma}_t}=\scal{\gamma_s}{R(\nabla_{\bar{\gamma}_t}\bar{\gamma}_t,\bar{\gamma}_t)\bar{\gamma}_t}$. 
As a consequence \eqref{firstvar} is equivalent to the following Euler-Lagrange equation
\begin{equation}\label{presque géodésique}
D_t\dot{u}=\sigma^2\left\{D_t\left(2D^2_t\dot{u}+3\kappa^2\dot{u}\right)+2R(D_t\dot{u},\dot{u})\dot{u}\right\}
\end{equation}
in the distributional sense. According to the forecoming part \ref{Passage}, this equation implies that $u$ is a $C^{\nu-1}$ function.

\section{Second Variation of Energy}

We recall that the second variation or Hessian is defined as follows. Let $u$ be a critical point of $E_\sigma$. For every $v\in\wt$, if $\gamma:I\times S^1$ is a $C^2$ variation of $u$ such that $\partial_s\gamma_{s=0}=v$, we define
\[
D^2E_\sigma(u)[v,v]=\left.\frac{\partial^2}{\partial s^2}E_{\sigma}(\gamma(s,\cdot))\right|_{s=0}
\]
and this definition is independent of the variation.

\begin{prop}
If $u\in\im$ is a critical point of $E_\sigma$, then for all $v\in \wt$, we have
\begin{align}\label{hessienne}
D^2E_{\sigma}(u)[v,v]&=2\sigma^2\int_0^{L}|D_t^2v|^2+|R(v,\dot{u})\dot{u}|^2+2\left(4\scal{D_tv}{\dot{u}}^2+2\scal{D_tv}{\dot{u}}-|D_tv|^2+\scal{R(\dot{u},v)v}{\dot{u}}\right)\kappa^2(u)\nonumber\\
&-\left(\scal{D_t^2v}{\dot{u}}+\scal{D_tv}{D_t\dot{u}}\right)^2-8\scal{D_tv}{\dot{u}}\scal{D_t^2v}{D_t\dot{u}}+\scal{\D_{\dot{u}}R(v,\dot{u})v}{D_t\dot{u}}\nonumber\\
&+\scal{\D_{v}R(v,\dot{u})\dot{u}}{D_t\dot{u}}+\scal{R(D_tv),\dot{u})v}{D_t\dot{u}}-\scal{R(D_t\dot{u},v)v}{D_t\dot{u}}+3\scal{R(v,\dot{u})D_tv}{D_t\dot{u}}\nonumber\\
&+\scal{R(\dot{u},D_tv)\dot{u}}{D_t\dot{u}}+2\scal{R(v,\dot{u})D_tv}{D_t\dot{u}}-6\scal{D_tv}{\dot{u}}\scal{R(v,\dot{u})\dot{u}}{D_t\dot{u}}d\leb^1\nonumber\\
&+4\sigma^2\int_0^{L}\scal{D_tv}{\dot{u}}\left(\scal{D_t^2v}{D_t\dot{u}}-2\scal{D_tv}{\dot{u}}\kappa^2(u)+\scal{R(v,\dot{u})\dot{u}}{D_t\dot{u}}\right)d\leb^1\nonumber\\
&+\int_0^{L}\left(1+\sigma^2\kappa^2(u)\right)\left(|D_tv|^2-\scal{D_tv}{\dot{u}}^2-\scal{R(\dot{u},v)v}{\dot{u}}\right)d\leb^1
\end{align}
\end{prop}
\begin{proof}
We may then choose a variation $\gamma$ such that
\begin{equation}\label{edo}
\left\{
\begin{aligned}
&D_s\partial_s\gamma=0\\
&\gamma(0,\cdot)=u\\
&\left.\partial_s\gamma\right|_{s=0}=v
\end{aligned}
\right.
\end{equation}
Indeed, as $u$ is critical point of $E_\sigma$, it is a 
$C^{\nu-1}$ function ($\nu-1\geq 2$). The Cauchy-Lipschitz theorem asserts the existence of a local $C^{\nu-1}$ function defined on an open neighbourhood of $\ens{0}\times S^1$ of this differential system. 

Let us denote with a slight change in the notations $\gamma_t=\dfrac{\partial_t\gamma}{|\partial_t\gamma_t|}$, $\gamma_s=\partial_s\gamma$, $\alpha=\scal{\D_{\gamma_t}\gamma_s}{\gamma_t}$. We will make constant use of the following identity
\begin{equation}\label{noncomm}
\D_{\gamma_s}\D_{\gamma_t}=\D_{\gamma_t}\D_{\gamma_s}+R(\gamma_s,\gamma_t)-\alpha\D_{\gamma_t}.
\end{equation}
which is a direct consequence of \ref{commute}, as $R$ is defined such that
\begin{align*}
R(\gamma_s,\gamma_t)=\D_{\gamma_s}\D_{\gamma_t}-\D_{\gamma_t}\D_{\gamma_s}-\D_{\lie{\gamma_s}{\gamma_t}}.
\end{align*}
As $\lie{\gamma_s}{\gamma_t}=-\alpha\gamma_t$, the preceding equation is equivalent to \eqref{noncomm}.

We shall also use the notations $D_t=\nabla_{\gamma_t}$, $D_s=\nabla_{\gamma_s}$. As a consequence \eqref{noncomm} reads
\[
D_sD_t=D_tD_s+R(\gamma_s,\gamma_t)-\scal{D_t\gamma_s}{\gamma_t}D_t
\]
with $\alpha=\scal{D_t\gamma_s}{\gamma_t}$. Finally, one has $[\gamma_s,\gamma_t]=-\alpha\gamma_t$, so in our new notation, this gives
\begin{align*}
D_s\gamma_t=D_t\gamma_s-\scal{D_t\gamma_s}{\gamma_t}\gamma_t=D_t\gamma_s-\alpha\gamma_t.
\end{align*}
Recall that 
\[
\kappa^2=\scal{D_t\gamma_t}{D_t\gamma_t}.
\]
We shall now proceed with the calculus of the second derivative of $\kappa^2$. By compatibility of the metric with $\D$, we have
\begin{align*}
\partial_s^2\kappa^2=2\scal{D_s^2D_t\gamma_t}{D_t\gamma_t}+2\scal{D_sD_t\gamma_t}{D_sD_t\gamma_t}=2\left\{(1)+(2)\right\}.
\end{align*}
now
\begin{align*}
D_sD_t\gamma_t&=D_tD_s\gamma_t+R(\gamma_s,\gamma_t)\gamma_t-\scal{D_t\gamma_s}{\gamma_t}D_t\gamma_t\\
&=D_t\left(D_t\gamma_s-\scal{D_t\gamma_s}{\gamma_t}\gamma_t\right)+R(\gamma_s,\gamma_t)\gamma_t-\scal{D_t\gamma_s}{\gamma_t}D_t\gamma_t\\
&=D_t^2\gamma_s-(\scal{D_t^2\gamma_s}{\gamma_t}+\scal{D_t\gamma_s}{D_t\gamma_t})\gamma_t-2\scal{D_t\gamma_s}{\gamma_t}D_t\gamma_t+R(\gamma_s,\gamma_t)\gamma_t\\
&=\mathrm{(I)}-\mathrm{(II)}-2\mathrm{(III)}+\mathrm{(IV)}.
\end{align*}
We split the computation into four parts.
\begin{align*}
D_s\mathrm{(I)}=D_sD_t^2\gamma_s&=D_tD_sD_t\gamma_s+R(\gamma_s,\gamma_t)D_t\gamma_s-\alpha D_t^2\gamma_s\\
&=D_t\left(D_tD_s\gamma_s+R(\gamma_s,\gamma_t)\gamma_s-\alpha D_t\gamma_s\right)+R(\gamma_s,\gamma_t)D_t\gamma_s-\alpha D_t^2\gamma_s\\
&=D_tR(\gamma_s,\gamma_t)\gamma_s+R(D_t\gamma_s,\gamma_t)\gamma_s+R(\gamma_s,D_t\gamma_t)\gamma_s+R(\gamma_s,\gamma_t)D_t\gamma_s\\
&\quad\quad\quad\quad\quad\quad\quad\;-(\partial_t\alpha)D_t\gamma_s-\alpha D_t^2\gamma_s+R(\gamma_s,\gamma_t)D_t\gamma_s-\alpha D_t^2\gamma_s\\
&=D_tR(\gamma_s,\gamma_t)\gamma_s+R(D_t\gamma_s,\gamma_t)\gamma_s+R(\gamma_s,D_t\gamma_t)\gamma_s+2R(\gamma_s,\gamma_t)D_t\gamma_s\\
&\quad\quad\quad\quad\quad\quad\quad\quad\quad\quad\quad\quad\quad\quad\;\,-(\partial_t\alpha)D_t\gamma_s-2\alpha D_t^2\gamma_s
\end{align*}

We recall the notation $\alpha=\scal{D_t\gamma_s}{\gamma_t}$. As $\mathrm{(II)}=\partial_t\alpha \gamma_t$,
\begin{align*}
D_s\mathrm{(II)}&=\partial_s\partial_t\alpha\gamma_t+\partial_t\alpha D_s\gamma_t\\
&=\partial_s\partial_t\alpha\gamma_t+\partial_t\alpha D_t\gamma_s-\alpha\partial_t\alpha\gamma_t
\end{align*}
Furthermore,
\begin{align*}
\partial_s\alpha&=\scal{D_sD_t\gamma_s}{\gamma_t}+\scal{D_t\gamma_s}{D_s\gamma_t}\\
&=\scal{D_tD_s\gamma_s+R(\gamma_s,\gamma_t)\gamma_s-\alpha D_t\gamma_s}{\gamma_t}+\scal{D_t\gamma_s}{D_t\gamma_s-\alpha\gamma_t}\\
&=|D_t\gamma_s|^2-2\alpha^2-\scal{R(\gamma_t,\gamma_s)\gamma_s}{\gamma_t}.
\end{align*}
Recalling that $\mathrm{(III)}=\alpha D_t\gamma_t$, one has
\begin{align*}
D_s\mathrm{(III)}&=\left(|D_t\gamma_s|^2-2\alpha^2-\scal{R(\gamma_t,\gamma_s)\gamma_s}{\gamma_t}\right)D_t\gamma_t+\alpha D_sD_t\gamma_t\\
&=\left(|D_t\gamma_s|^2-2\alpha^2-\scal{R(\gamma_t,\gamma_s)\gamma_s}{\gamma_t}\right)D_t\gamma_t
+\alpha\left\{D_t^2\gamma_s-(\partial_t\alpha)\gamma_t-2\alpha D_t\gamma_t+R(\gamma_s,\gamma_t)\gamma_t\right\}\\
&=\left(|D_t\gamma_s|^2-4\alpha^2-\scal{R(\gamma_t,\gamma_s)\gamma_s}{\gamma_t}\right)D_t\gamma_t
+\alpha\left\{D_t^2\gamma_s-(\partial_t\alpha)\gamma_t+R(\gamma_s,\gamma_t)\gamma_t\right\}.
\end{align*}
According to the defining properties of the Riemannian curvature tensor $R$, we have
\begin{align*}
D_s\mathrm{(IV)}=D_s(R(\gamma_s,\gamma_t)\gamma_t)&=D_sR(\gamma_s,\gamma_t)\gamma_t+R(D_s\gamma_s,\gamma_t)\gamma_t+R(\gamma_s,D_s\gamma_t)\gamma_t+R(\gamma_s,\gamma_t)D_s\gamma_t
\\&=D_sR(\gamma_s,\gamma_t)\gamma_t+R(\gamma_s,D_t\gamma_s)\gamma_t+R(\gamma_s,\gamma_t)D_t\gamma_s-2\alpha R(\gamma_s,\gamma_t)\gamma_t
\end{align*}
as $D_s\gamma_t=D_t\gamma_s-\alpha\gamma_t$.
If we parametrise $u$ in arc-length, then
\begin{align*}
\scal{\gamma_t}{D_t\gamma_t}=0.
\end{align*}
In $s=0$, we have
\begin{align*}
\scal{D_s\mathrm{(II)}}{D_t\gamma_t}=\partial_t\alpha\scal{D_t\gamma_s}{D_t\gamma_t}=\left(\scal{D_t^2\gamma_s}{\gamma_t}+\scal{D_t\gamma_s}{D_t\gamma_t}\right)\scal{D_t\gamma_s}{D_t\gamma_t}
\end{align*}
and
\begin{align*}
\scal{D_s\mathrm{(III)}}{D_t\gamma_t}&=\left(|D_t\gamma_s|^2-2\alpha^2-\scal{R(\gamma_t,\gamma_s)\gamma_s}{\gamma_t}\right)\kappa^2+
\scal{D_t\gamma_s}{\gamma_t}\scal{D_t^2\gamma_s}{D_t\gamma_t}.
\end{align*}
The first term of the second derivative is
\begin{align*}
(1)&=\scal{D_tR(\gamma_s,\gamma_t)\gamma_s+R(D_t\gamma_s,\gamma_t)\gamma_s+R(\gamma_s,D_t\gamma_t)\gamma_s+2R(\gamma_s,\gamma_t)D_t\gamma_s}{D_t\gamma_t}\tag*{(I)}\\
&-\left(\scal{D_t^2\gamma_s}{\gamma_t}+\scal{D_t\gamma_s}{D_t\gamma_t}\right)\scal{D_t\gamma_s}{D_t\gamma_t}-2\scal{D_t\gamma_s}{\gamma_t}\scal{D_t^2\gamma_s}{D_t\gamma_t}\tag*{(I)}\\
&-\left(\scal{D_t^2\gamma_s}{\gamma_t}+\scal{D_t\gamma_s}{D_t\gamma_t}\right)\scal{D_t\gamma_s}{D_t\gamma_t}\tag*{(II)}\\
&-2\left(|D_t\gamma_s|^2-4\scal{D_t\gamma_s}{\gamma_t}^2-\scal{R(\gamma_t,\gamma_s)\gamma_s}{\gamma_t}\right)\kappa^2-2\scal{D_t\gamma_s}{\gamma_t}\scal{D_t^2\gamma_s}{D_t\gamma_t}\tag*{(III)}\\
&+\scal{D_sR(\gamma_s,\gamma_t)\gamma_t}{D_t\gamma_t}+\scal{R(\gamma_s,D_t\gamma_s)\gamma_t}{D_t\gamma_t}+\scal{R(\gamma_s,\gamma_t)D_t\gamma_t}{D_t\gamma_s}\tag*{(IV)}\\
&-2\scal{D_t\gamma_s}{\gamma_t}\scal{R(\gamma_s,\gamma_t)\gamma_t}{D_t\gamma_t}\tag*{(IV)}
\end{align*}
while
\begin{align*}
(2)&=\left|D_t^2\gamma_s-(\scal{D_t^2\gamma_s}{\gamma_t}+\scal{D_t\gamma_s}{D_t\gamma_t})\gamma_t-2\scal{D_t\gamma_s}{\gamma_t}D_t\gamma_t+R(\gamma_s,\gamma_t)\gamma_t\right|^2\\
&=|D_t^2\gamma_s|^2+\left(\scal{D_t^2\gamma_s}{\gamma_t}+\scal{D_t\gamma_s}{D_t\gamma_t}\right)^2+4\scal{D_t\gamma_s}{\gamma_t}\kappa^2+|R(\gamma_s,\gamma_t)\gamma_t|^2\\
&-2\left(\scal{D_t^2\gamma_s}{\gamma_t}+\scal{D_t\gamma_s}{D_t\gamma_t}\right)\scal{D_t^2\gamma_s}{\gamma_t}-4\scal{D_t\gamma_s}{\gamma_t}\scal{D_t^2\gamma_s}{D_t\gamma_t}+2\scal{R(\gamma_s,\gamma_t)\gamma_t}{D_t^2\gamma_s}\\
&-4\scal{D_t\gamma_s}{\gamma_t}\scal{R(\gamma_s,\gamma_t)\gamma_t}{D_t\gamma_t}
\end{align*}
We deduce that in $s=0$, we have
\begin{align*}
\partial_s^2\kappa^2&=2|D_t^2\gamma_s|^2+2|R(\gamma_s,\gamma_t)\gamma_t|^2
+4\left(4\scal{D_t\gamma_s}{\gamma_t}^2+2\scal{D_t\gamma_s}{\gamma_t}-|D_t\gamma_s|^2+\scal{R(\gamma_t,\gamma_s)\gamma_s}{\gamma_t}\right)\kappa^2\\
&-2\left(\scal{D_t^2\gamma_s}{\gamma_t}
+\scal{D_t\gamma_s}{D_t\gamma_t}\right)^2
-16\scal{D_t\gamma_s}{\gamma_t}\scal{D_t^2\gamma_s}{D_t\gamma_t}\\
&+2\scal{D_tR(\gamma_s,\gamma_t)\gamma_s}{D_t\gamma_t}+
2\scal{D_sR(\gamma_s,\gamma_t)\gamma_t}{D_t\gamma_t}+2\scal{R(D_t\gamma_s,\gamma_t)\gamma_s}{D_t\gamma_t}\\
&+2\scal{R(\gamma_s,D_t\gamma_t)\gamma_s}{D_t\gamma_t}
+6\scal{R(\gamma_s,\gamma_t)D_t\gamma_s}{D_t\gamma_t}
+2\scal{R(\gamma_t,D_t\gamma_s)\gamma_t}{D_t\gamma_t}\\
&+4\scal{R(\gamma_s,\gamma_t)\gamma_t}{D_t^2\gamma_s}
-12\scal{D_t\gamma_s}{\gamma_t}\scal{R(\gamma_s,\gamma_t)\gamma_t}{D_t\gamma_t}
\end{align*}

Now
\begin{align*}
\partial_s\kappa^2=2\scal{D_t^2v}{D_t\dot{u}}-4\scal{D_tv}{\dot{u}}\kappa^2+2\scal{R(v,\dot{u})\dot{u}}{D_t\dot{u}},
\end{align*}
and
\begin{align*}
&\int_{S^1}\left(1+\sigma^2\kappa^2\right)\partial_s^2|\gamma_t|d\leb^1=\int_{S^1}\left(1+\sigma^2\right)\partial_s\scal{\D_{\gamma_t}\gamma_s}{\frac{\gamma_t}{|\gamma_t|}}d\leb^1\\
&=\int_{S^1}\left(1+\sigma^2\kappa^2\right)\left\{\scal{\D_{\gamma_s}D_{\gamma_t}\gamma_s}{\frac{\gamma_t}{|\gamma_t|}}+\scal{\D_{\gamma_t}\gamma_s}{\frac{\D_{\gamma_s}\gamma_t}{|\gamma_t|}}-\scal{\D_{\gamma_t}\gamma_s}{-\scal{\D_{\gamma_t}\gamma_s}{\gamma_t}\frac{\gamma_t}{|\gamma_t|^3}}\right\}d\leb^1\\
&=\int_{S^1}(1+\sigma^2\kappa^2)\left\{\scal{\D_{\gamma_t}\D_{\gamma_s}\gamma_s}{\frac{\gamma_t}{|\gamma_t|}}+\scal{R(\gamma_s,\gamma_t)\gamma_s}{\frac{\gamma_t}{|\gamma_t|}}+|\D_{\frac{\gamma_t}{|\gamma_t|}}\gamma_s|^2|\gamma_t|-\scal{\D_{\frac{\gamma_t}{|\gamma_t|}}\gamma_s}{\gamma_t}^2|\gamma_t|\right\}d\leb^1.
\end{align*}

As $D_s\gamma_s=0$, at $s=0$, the preceding equation is equal to
\begin{align*}
\int_{0}^L\left(1+\sigma^2\kappa^2(u)\right)\left\{|D_tv|^2-\scal{D_tv}{\dot{u}}^2-\scal{R(\dot{u},v)v}{\dot{u}}\right\}d\leb^1.
\end{align*}
Furthermore
\[
\frac{d^2}{ds^2}E_\sigma(\gamma(s,\cdot))=\int_{S^1}\sigma^2\partial_s^2\kappa^2|\gamma_t|+2\sigma^2\partial_s\kappa^2\partial_s|\gamma_t|+(1+\sigma^2\kappa^2)\partial_s^2|\gamma_t|d\leb^1
\]

Finally, we deduce that
\begin{align}
D^2E_{\sigma}(u)[v,v]&=2\sigma^2\int_0^{L}|D_t^2v|^2+|R(v,\dot{u})\dot{u}|^2+2\left(4\scal{D_tv}{\dot{u}}^2+2\scal{D_tv}{\dot{u}}-|D_tv|^2+\scal{R(\dot{u},v)v}{\dot{u}}\right)\kappa^2(u)\nonumber\\
&-\left(\scal{D_t^2v}{\dot{u}}+\scal{D_tv}{D_t\dot{u}}\right)^2-8\scal{D_tv}{\dot{u}}\scal{D_t^2v}{D_t\dot{u}}+\scal{\D_{\dot{u}}R(v,\dot{u})v}{D_t\dot{u}}\nonumber\\
&+\scal{\D_{v}R(v,\dot{u})\dot{u}}{D_t\dot{u}}+\scal{R(D_tv),\dot{u})v}{D_t\dot{u}}-\scal{R(D_t\dot{u},v)v}{D_t\dot{u}}+3\scal{R(v,\dot{u})D_tv}{D_t\dot{u}}\nonumber\\
&+\scal{R(\dot{u},D_tv)\dot{u}}{D_t\dot{u}}+2\scal{R(v,\dot{u})D_tv}{D_t\dot{u}}-6\scal{D_tv}{\dot{u}}\scal{R(v,\dot{u})\dot{u}}{D_t\dot{u}}d\leb^1\nonumber\\
&+4\sigma^2\int_0^{L}\scal{D_tv}{\dot{u}}\left(\scal{D_t^2v}{D_t\dot{u}}-2\scal{D_tv}{\dot{u}}\kappa^2(u)+\scal{R(v,\dot{u})\dot{u}}{D_t\dot{u}}\right)d\leb^1\nonumber\\
&+\int_0^{L}\left(1+\sigma^2\kappa^2(u)\right)\left(|D_tv|^2-\scal{D_tv}{\dot{u}}^2-\scal{R(\dot{u},v)v}{\dot{u}}\right)d\leb^1
\end{align}
which concludes the proof of the proposition.
\end{proof}
We will use later the result to investigate the index of the curves in section \ref{indice}.

\section{Palais-Smale Condition}

We recall the definition of the Palais-Smale condition.
\begin{defi}
Let $X$ be a Finsler $C^\nu$ manifold ($\nu\in\N\cup\ens{\infty})$, and $f\in C^1(X)$. We say that $f$ satisfies the Palais-Smale condition at the level $c\in\R$ if for every sequence $\ens{x_n}_{n\in\N}\subset X$, if 
\begin{align*}
f(x_n)\conv c,\quad \text{and}\quad Df(x_n)\conv 0
\end{align*}
then there exist a subsequence of $\ens{x_n}_{n\in \N}$ (strongly) converging towards an element $x\in X$.
\end{defi}

The main result of this section is the following (contained in another closed form in \cite{LZ2}).
\begin{theorem}\label{palaissmale}
Let $\sigma,c>0$ two positive real numbers. Then $E_\sigma$ satisfies the Palais-Smale function at the level $c$.
\end{theorem}
\begin{proof}
Let $\ens{u_n}_{n\in\N}\subset \im$ a sequence such that
\begin{align*}
E(u_n)\conv c,\quad\text{and}\quad DE(u_n)\conv 0
\end{align*}
The second hypothesis should be interpreted as
\begin{equation}
\lim_{n\rightarrow\infty}\sup\ens{DE_\sigma(u)\cdot v,\;\nwp{v}{2,2}{S^1}\leq 1}=0
\end{equation}
where we recall that
\begin{equation}
DE_\sigma(u)\cdot v=\int_{S^1}\left(2\sigma^2\scal{D_t\dot{u}_n}{D_t^2v}+(1-3\sigma^2\kappa^2(u_n))\scal{\dot{u_n}}{D_t v}+2\sigma^2\scal{R(D_t\dot{u}_n,\dot{u}_n)\dot{u}_n}{D_t v}\right)|\dot{u}|d\leb^1
\end{equation}
In particular, there exists a constant $\epsilon>0$ such that
\[
L_n=\mathfrak{L}(u_n)=\int_{S^1}|\dot{u}_n|d\leb^1\geq  \epsilon c,\;\text{et}\; E_\sigma(u_n)\leq 2c
\]
for all $n$ great enough. We may then assume this property for all $n\in\N$.
As the manifold $(M,g)$ is compact, 
\[
\sup_{n\in\N}\np{u_n}{\infty}<\infty,
\]
so obviously
\[
\sup_{n\in\N}\np{u_n}{2}<\infty.
\]
Furthermore, the periodicity of $u_n$ implies that
\[
\int_{S^1}\dot{u_n}\,d\leb^1=0,
\]
so according to Poincaré-Wirtinger inequality, 
\begin{align*}
\int_{0}^{L_n}|\dot{u}_n|^2d\leb^1\leq \left(\frac{L_n}{2\pi}\right)^2\int_{0}^{L_n}|D_t\dot{u}_n|^2d\leb^1=\left(\frac{L_n}{2\pi}\right)^2\int_{S^1}\kappa^2(u_n)|\dot{u}_n|d\leb^1\leq \left(\frac{L_n}{2\pi}\right)^2\frac{2c}{\pi\sigma^2}\leq \frac{2c^3}{\pi^2\sigma^2}
\end{align*}
thus
\begin{align*}
\np{\dot{u_n}}{2}^2&=\int_{S^1}|\dot{u}_n|^2d\leb^1=\int_{0}^{L_n}|\dot{u}_n|d\leb^1
\leq \sqrt{L_n}{\left(\int_0^{L_n}|\dot{u}_n|^2d\leb^1\right)}^{\frac{1}{2}}
\leq \frac{2c^2}{\pi\sigma}.
\end{align*}
We deduce that
\begin{equation}
\sup_{n\in\N}\nwp{u_n}{2,2}{S^1}<\infty.
\end{equation}
By Cauchy-Schwarz inequality, for all $(x,y)\in S^1\times S^1$,
\begin{align*}
|u_n(x)-u_n(y)|&=\left|\int_{S^1}\dot{u}_n(t)d\leb^1t\right|\\
&\leq \sqrt{|x-y|}\np{\dot{u}_n}{2}
\end{align*}
so the sequence $\ens{u_n}_{n\in\N}$ is equicontinuous, and likewise
\begin{equation}
|\dot{u}_n(x)-\dot{u}_n(y)|\leq \sqrt{|x-y|}\np{D_t\dot{u}_n}{2}.
\end{equation}
By Sobolev embedding theorem, $\W{2,2}{S^1}\hookrightarrow C^{1,\frac{1}{2}}(S^1)$, so according to Arzelà–Ascoli and Rellich-Kondrachov theorems, there exist a subsequence of $\ens{u_n}_{n\in\N}$ (still denoted $\ens{u_n}_{n\in\N}$), and a function $u\in \W{2,2}{S^1}$ such that
\begin{align}\label{convergence}
&u_n\rightharpoonup u\quad \text{weakly in}\;\,\W{2,2}{S^1}\nonumber\\
&u_n\rightarrow u\quad\text{strongly}\;\, \lp{\infty}{S^1}\\
&\dot{u}_n\rightarrow \dot{u}\quad \text{strongly}\;\, \lp{\infty}{S^1}\nonumber
\end{align}
In particular,
\begin{equation}\label{w12}
\lim_{n\rightarrow\infty}\nwp{u_n-u}{1,2}{S^1}=0.
\end{equation}
Furthermore, assume for one moment that $\ens{\dot{u}_n}_{n\in\N}$ is given in normal parametrization, such that $|\dot{u}_n|=1$ on $[0,L_n]$. Then, thanks of \eqref{convergence}, if $L=\mathfrak{L}(u)$, we have 
\begin{align*}
	L_n\conv L
\end{align*}
and by uniform convergence of $\ens{\dot{u}_n}_{n\in\N}$ towards $\dot{u}$, we have $|\dot{u}|=1$ on $[0,L]$, so $u$ is a $\mathrm{W}^{1,2}(S^1,M)$ immersion. \vspace{0.2em}By compactness of $S^1$, we deduce that $\ens{D_t\dot{u}_n}_{n\in\N}$ converge in $\mathrm{L}^2(S^1,M)$ to $D_t\dot{u}$ if and only if
\begin{align*}
	\int_{S^1}|P(u_n)(\ddot{u}_n-\ddot{u})|^2|\dot{u}_n|d\leb^1\conv 0.
\end{align*}
As $\ens{u_n}_{n\in\N}$ is bounded in $\im$, we deduce that if $v_n=P(u_n)(u_n-u)$, then
\begin{align*}
\lim_{n\rightarrow\infty}DE_\sigma(u_n)\cdot v_n =0
\end{align*} 
and
\begin{align}
\label{e1}&DE_\sigma(u_n)\cdot v_n=\int_{S^1}\left(2\sigma^2\scal{D_t\dot{u}_n}{P(u_n)(\ddot{u}_n-\ddot{u})+2 DP(u_n)(\dot{u}_n)(\dot{u}_n-\dot{u})
+D^2P(u_n)(\dot{u}_n,\dot{u}_n)(u_n-u))}\right.\\
\label{e2}&+(1-3\sigma^2\kappa^2(u_n))\scal{\dot{u}_n}{P(u_n)(\dot{u_n}-\dot{u})+DP(u_n)(\dot{u}_n)(u_n-u)}\\
\label{e3}&\left.+2\scal{R(D_t\dot{u}_n,\dot{u}_n)\dot{u_n}}{P(u_n)(u_n-u)}\right)|\dot{u}_n| d\leb^1.
\end{align}
As $(M,g)$ is a $C^\nu$ compact manifold, and $P$ is $C^{\nu-1}$ and $\ens{u_n}_{n\in\N}$ is bounded in $\mathrm{W}^{1,\infty}(S^1,M)$. This ensures the existence of a constant $C>0$ independent of $n\in\N$ such that
\begin{align*}
\np{DP(u_n)(\dot{u}_n-\dot{u})}{\infty}&\leq C\np{\dot{u}_n}{\infty}\np{\dot{u}_n-\dot{u}}{\infty}\\
\np{DP(u_n)(\dot{u}_n)(u_n-u)}{\infty}&\leq C\np{\dot{u}_n}{\infty}\np{u_n-u}{\infty}\\
\np{DP(u_n)(\dot{u}_n)(\dot{u}_n-\dot{u})}{2}&\leq C\np{\dot{u}_n}{\infty}\np{\dot{u}_n-\dot{u}}{2}\\
\np{D^2P(u_n)(\dot{u}_n,\dot{u}_n)(u_n-u)}{2}&\leq C\np{\dot{u}_n}{\infty}^2\np{u_n-u}{2}.
\end{align*}

Now $\ens{D_t\dot{u}_n}_{n\in\N}$ is bounded in $\mathrm{L}^2(S^1,TM)$ so by Cauchy-Schwarz inequality,
\begin{align*}
&\left|\int_{S^1}\scal{D_t\dot{u}_n}{2DP(u_n)(\dot{u}_n)(\dot{u}_n-\dot{u})+D^2P(u_n)(\dot{u}_n,\dot{u}_n)(u_n-u)}|\dot{u}_n|d\leb^1\right|\\
&\leq C\np{D_t\dot{u}_n}{2}\np{\dot{u}_n}{\infty}^2\left(2\np{\dot{u}_n-\dot{u}}{2}+\np{\dot{u}_n}{\infty}\np{u_n-u}{2}\right)\conv 0
\end{align*}

We can estimate \eqref{e2} as follows:
\begin{align*}
&\left|\int_{S^1}\left(1+3\sigma^2\kappa^2(u_n)\right)\scal{\dot{u}_n}{P(u_n)(\dot{u_n})(u_n-u)+DP(u_n)(\dot{u}_n)(\dot{u}_n-\dot{u})}|\dot{u}_n|d\leb^1\right|\\
&\leq C\left(1+3\sigma^2\np{D_t\dot{u}_n}{2}^2\right)\np{\dot{u}_n}{2}^3\left(\np{\dot{u}_n-\dot{u}}{2}+\np{u_n-u}{2}\right)\conv 0.
\end{align*}

Finally, the metric $g$ is $C^\nu$, so the $(3,1)$-curvature tensor $R$ is $C^{\nu-2}$, its components are bounded on the compact manifold $(M,g)$ in the following sense: if we write
\begin{align*}
\int_{S^1}\scal{R(D_t\dot{u}_n,\dot{u}_n)\dot{u_n}}{P(u_n)(u_n-u)}|\dot{u}_n|d\leb^1&=
\int_{S^1}\sum_{i,j,k,l=1}^n R_{i,j,k}^lD_t\dot{u}_n^i\dot{u}_n^j\dot{u}_n^k(P(u_n)(u_n-u))^l\,|\dot{u}_n|d\leb^1.
\end{align*}
and define $\displaystyle\np[M]{R}{\infty}=\sup_{1\leq i,j,k,l\leq n}\np[M]{R_{i,j,k}^l}{\infty}$, then $\np[M]{R}{\infty}<\infty$, and
\begin{align*}
&\left|\int_{S^1}\scal{R(D_t\dot{u}_n,\dot{u}_n)\dot{u_n}}{P(u_n)(u_n-u)}|\dot{u}_n|d\leb^1\right|\\
&\leq \np{R}{\infty}\np{D_t\dot{u}_n}{2}{\left(\int_{S^1}|\dot{u}_n|^6|P(u_n)(u_n-u)|^2d\leb^1\right)}^{\frac{1}{2}}\\
&\leq C\np{R}{\infty} \np{D_t\dot{u}_n}{2}\np{\dot{u}_n}{\infty}^3\np{u_n-u}{2}\conv 0.
\end{align*}

The first member of \eqref{e1} is equal to
\begin{align*}
\int_{S^1}\scal{D_t\dot{u}_n}{P(u_n)(\ddot{u}_n-\ddot{u})}|\dot{u}_n|d\leb^1&=\int_{S^1}\left(|P(u_n)(\ddot{u}_n-\ddot{u})|^2+\scal{P(u_n)\ddot{u}}{P(u_n)(\ddot{u}_n-\ddot{u})}\right)|\dot{u}_n|d\leb^1
\end{align*}
As $\ens{D_t\dot{u}_n}_{n\in\N}$ converges weakly towards $D_t\dot{u}$ in $\mathrm{L}^2(S^1,TM)$, and $\ens{P(u_n)\ddot{u}}_{n\in\N}$ is bounded on $\mathrm{L}^2(S^1,TM)$, we have
\[
\lim_{n\rightarrow \infty}\int_{S^1}\scal{D_tu}{P(u_n)(\ddot{u}_n-\ddot{u})}|\dot{u}_n|d\leb^1=0.
\]
We finally deduce that
\[
\lim_{n\rightarrow\infty}\int_{S^1}|D_t\dot{u}_n-P(u_n)\ddot{u}|^2|\dot{u}_n|d\leb^1=0,
\]
so thanks of \eqref{convergence}
\[
\lim_{n\rightarrow\infty}\nwp{u_n-u}{2,2}{S^1}=0.
\]
This concludes the proof of the theorem.
\end{proof}

\section{Min-Max Construction of Adapted Sequence of Critical Points}\label{bonnesuite}

We aim in this section as constructing a sequence of critical points $\ens{u_n}_{n\in\N}$ associated to $\ens{\sigma_n}_{n\in\N}$, where $\ens{\sigma_n}_{n\in\N}$ is a sequence of positive numbers converging to $0$, such that
\[
\sigma_n^2\int_{S^1}\kappa^2(u_n)|\dot{u}_n|d\leb^1\conv 0.
\]
The principle of proof is adapted from  a result of Michael Struwe (see \cite{struweart}).
\begin{defi}\label{admissible}
	A family of non-empty sets $\A\subset\P^*(\mathrm{W}^{2,2}(S^1,M))$ is admissible if the three following conditions are realised by  $\A$:
	\begin{itemize}
		\itemsep-0.8em 
		\item[1)] For all $A\in\A$, for all $u\in\A$, either $u$ is a constant curve either $u\in\im$,\\
		\item[2)] For every homemorphism $\phi$ of $\im$ isotopic to the identity map, for all $A\in\A$, $\phi(A)\in A$,\\
		\item[3)] There exists a positive integer $k\in\N$ such that for all $A\in\A$, we can write $A=\ens{u_t^A}_{t\in [0,1]^k}$, and the map
		\vspace{-0.5em}
		\begin{align*}
		[0,1]^k&\rightarrow\mathrm{W}^{2,2}(S^1,M)\\
		t&\mapsto u_t^A
		\end{align*}
		\vspace{-0.5em}
		is continuous.		
	\end{itemize}
\end{defi}

We now fix an admissible set $\mathscr{A}\subset\mathscr{P}^*\left(\mathrm{W}^{2,2}(S^1,M)\right)$ such that
\begin{align*}
0<\beta(0)=\inf_{A\in\mathscr{A}}\sup \mathfrak{L}(A)<\infty.
\end{align*}
We then define the family $\mathscr{A}_0\in\P^*(\im)$ from $\mathscr{A}$ by
\begin{align*}
\mathscr{A}_0=\ens{A_{0}, A\in\mathscr{A}\;\text{et}\; A_0\neq\varnothing},
\end{align*}
where for all $A\in\mathscr{A}$, 
\begin{align*}
A_0=A\cap\ens{u:\mathfrak{L}(u)\geq \frac{\beta(0)}{2}}.
\end{align*}
We remark that if $\phi$ is an homeomorphism of $\im$ isotopic to the identity, in genral, $\phi(A_0)\neq \phi(A)_0$.
For all $\sigma>0$, we define
\[
\beta(\sigma)=\inf_{A_0\in\mathscr{A}_0} \sup E_\sigma(A_0)<\infty
\]
Indeed, the function $\mathfrak{L}$ is continuous on $\mathrm{W}^{2,2}(S^1,M)$, and for all $A_0\in \mathscr{A}_0$, $A_0=\ens{u_t}_{t\in I}$, where $I$ is a closed subset of $[0,1]^k$, so the application $I\rightarrow \im$, $t\rightarrow u_t$ is continuous thus
\begin{align*}
\sup E_\sigma(A_0)=\sup_{t\in I}E_\sigma(u_t)<\infty
\end{align*}
and $\beta(\sigma)<\infty$. 

We now observe that
\begin{align*}
\beta(\sigma)\conv \beta(0).
\end{align*}
To prove this claim, remark that for all $\sigma>0$, $\beta(\sigma)\geq \beta(0)$ so by definiition of $\beta(0)$, for all $\epsilon>0$, there exists $A\in\mathscr{A}$, such that
\begin{align*}
\sup_{u\in A}\mathfrak{L}(u)<\beta(0)+\epsilon.
\end{align*}
Therefore
\begin{align*}
\sup_{u\in A_0}E_\sigma(u)\leq \beta(0)+\epsilon+\sigma^2\sup_{u\in A_0}\int_{S^1}\kappa^2(u)d\leb^1\leq \beta(0)+\epsilon+C\sigma^2
\end{align*}
so for $C\sigma^2\leq \epsilon$, we have
\begin{align*}
\beta(0)\leq \beta(\sigma)\leq \beta(0)+2\epsilon
\end{align*}
and $\beta$ is increasing, so the claim is proved.

As $\beta$ is monotone, Lebesgue theorem ensures that this real function is differentiable $\leb^1$ almost everywhere. In particular,
\begin{equation}\label{limlog}
\delta=\liminf_{\sigma\rightarrow 0}\left(\sigma\log\frac{1}{\sigma}\beta'(\sigma)\right)=0.
\end{equation}
Let us argue by contradiction. If $\delta>0$, then for all $\sigma>0$ small enough, we have
\begin{align*}
\beta(\sigma)-\beta(0)\geq \int_0^\sigma \beta'(s)ds\geq \frac{1}{2}\delta\int_0^\sigma \frac{ds}{s\log\frac{1}{s}}=\infty,
\end{align*}
which gives the contradiction. 

These observations allow to introduce the following definition.
\begin{defi} Let $\sigma>0$ a fixed positive real number. We say that the function $\beta$ satisfies the entropy condition at a point $\sigma$ if it is differentiable at $\sigma$ and
	\begin{align}\label{entropie}
	\beta'(\sigma)\leq \frac{1}{\sigma\log\frac{1}{\sigma}}.
	\end{align}
\end{defi}

A formal derivation under the min-max would give a sequence of positive $\ens{\sigma_n}_{n\in\N}$ converging to $0$, and a sequence of critical points $\ens{u_{n}}_{n\in\N}$ associated to $\ens{\sigma_n}_{n\in \N}$, such that 
\[
\sigma_n\log\frac{1}{\sigma_n}\frac{dE_{\sigma_n}}{d\sigma}(u_{n})\rightarrow 0\; \text{when}\; n\rightarrow\infty,
\] 
which in turn would imply that
\[
\lim_{n\rightarrow\infty}\sigma_n^2\int_{S^1}\kappa^2(u_n)|\dot{u}_n|d\leb^1=0.
\]

The preceding intuition can be made rigourous thanks of the following proposition.
\begin{prop} There exists a constant $C=C(\beta(0))$, such that for all $0<\sigma\leq C(\beta(0))$ for which $\beta$ satisfies the entropy condition \eqref{entropie},  there exists a critical point $u_\sigma\in\im$ of $E_\sigma$ such that
	\begin{align}
	E_\sigma(u_\sigma)=\beta(\sigma),\quad\text{et}\quad \partial_\sigma E_\sigma(u_\sigma)\leq \beta'(\sigma)+\frac{1}{\sigma\log\frac{1}{\sigma}}.
	\end{align}
\end{prop}
\begin{proof}
	\textbf{Step 1} : estimation of the derivative of $E_\sigma$.\newline
	Let $\epsilon>0$ a positive fixed constant. We consider a sequence $\ens{\sigma_n}_{n\in\N}$ strictly decreasing to $\sigma>0$. 
	Let $A_0\in\mathscr{A}_0$ and $u\in A_0$, such that 
	\[
	E_\sigma(u)\geq \beta(\sigma)-\epsilon(\sigma_n-\sigma)
	\]
	and
	\[
	E_{\sigma_n}(u)\leq \beta(\sigma_n)+\epsilon(\sigma_n-\sigma).
	\]
	Such a pair $(u,A_0)$ always exists, for $n$ large enough. As $\beta$ is differentiable at $\sigma$, we have
	\[
	\beta(\sigma_n)\leq \beta(\sigma)+(\beta'(\sigma)+\epsilon)(\sigma_n-\sigma)
	\]
	for $n$ large enough, from which we deduce that
	\begin{equation}\label{intervallePS}
	\beta(\sigma)-\epsilon(\sigma_n-\sigma)\leq E_\sigma(u)\leq E_{\sigma_n}(u)\leq \beta(\sigma)+(\beta'(\sigma)+2\epsilon)(\sigma_n-\sigma)
	\end{equation}
	If $u$ satisfies \eqref{intervallePS}, then
	\begin{align*}
	\frac{E_{\sigma_n}(u)-E_\sigma(u)}{\sigma_n-\sigma}\leq \beta'(\sigma)+3\epsilon
	\end{align*}
	so according to the mean value theorem, there exists $\sigma'\in[\sigma,\sigma_n]$, such that
	\[
	\partial_\sigma E_{\sigma'}(u)\leq \beta'(\sigma)+3\epsilon.
	\]
	But
	\begin{align*}
	\partial_\sigma E_{\sigma'}(u)&=\int_{S^1}2\sigma'\kappa^2(u)|\dot{u}|d\leb^1=\frac{\sigma'}{\sigma}\partial_\sigma E_\sigma(u)
	\end{align*}
	so for all $u\in\im$ satisfying the inequalities \eqref{intervallePS}, 
	\begin{equation}\label{entropiegratuite}
	\partial_\sigma E_\sigma(u)\leq \beta'(\sigma)+3\epsilon.
	\end{equation}

	\textbf{Step 2}: existence of almost Palais-Smale sequences.
	
	We want to show that there exists a sequence $\ens{u_n}_{n\in\N}$ satisfying \eqref{intervallePS} and such that
	\begin{equation}\label{suitepalaissmale}
	\Vert DE_{\sigma_n}(u_n)\Vert\underset{n\rightarrow\infty}{\longrightarrow}0
	\end{equation}
	We shall be careful to distonguish this condition from the Palais-Smale condition for $E_\sigma$, but we will show in the next step that it implies Palais-Smale condition for $E_\sigma$. 
	
	We argue by contradiction, supposing the existence of a positive constant $\delta>0$ such that for all immersion $u\in\im$ satisfying \eqref{intervallePS}, we have for $n$ large enough
	\[
	\Vert DE_{\sigma_n}(u)\Vert\geq \delta.
	\]
	Let $X^0_n$ a pseudo-gradient vector field (see \cite{struwevar}.) for $E_{\sigma_n}$, \textit{i.e.} a locally Lipschitz bounded function $X_n^0:\im\rightarrow T\im$, such that for all $w\in \im$ such that $DE_{\sigma_n}(w)\neq 0$,
	\begin{align*}
	\Vert X_n^0(w)\Vert&<2 \min\ens{\Vert DE_{\sigma_n}(w)\Vert,1}\\
	DE_{\sigma_n}(w)\cdot X_n^0(w)&>\min\ens{\Vert DE_{\sigma_n}(w)\Vert,1}\Vert DE_\sigma(w)\Vert.
	\end{align*}
	
	Let $\psi\in \mathscr{D}(\R)$ a positive non-decreasing cut-off function such that $0\leq \psi\leq 1$, $\supp\psi\subset \R_+$, and $\psi=1$ on $[1,\infty[$. We define \footnote{Note that we use a different cut-off function from the original paper \cite{struweart}.} for all $n\in\N$,
	\[
	\psi_n(u)=\psi\left(\frac{E_{\sigma}(u)-(\beta(\sigma)-\epsilon(\sigma_n-\sigma))}{\epsilon(\sigma_n-\sigma)}\right)\psi\bigg(\frac{4}{\beta(0)}\frac{e^{-\frac{4}{\beta(0)}}}{\sigma}\left(\mathfrak{L}(u)-\frac{\beta(0)}{2}\right)\bigg).
	\]
	Let $\phi_n$ the global flow associated to $-X_n=-\psi_n X_n^0$, defined by
	\begin{equation}\label{eqdiffrem}
	\left\{
	\begin{aligned}
	\frac{d}{dt}\phi_n^t(u)&=-X_n(\phi_n^t(u))\\
	\phi_n^0(u)&=u
	\end{aligned}\right.
	\end{equation}
	Note that $\phi_n$ is $C^{1}$ from respect of the first variable, and that for all  $t\in\R_+$, $\phi_n^t:\im\rightarrow\im$ is a locally Lipschitz homeomorphism.
	We remark that $\A_0$ is invariant under the action of $\phi_n$, and that for all $A_0\in\mathscr{A}_0$, for all $t\geq 0$, $\phi_n^t(A_0)=\phi_n^t(A_0)_0$. We have 
	\begin{align*}
	\frac{d}{dt}E(\phi_n^t(u))=-\psi_n(u)DE_{\sigma_n}(u)\cdot X_n(u)\leq -\delta\psi_n(u)\leq 0
	\end{align*}
	We would like to show that $t\mapsto E_\sigma(\phi_n^t(u))$ is also decreasing. Consider, $u\in \im$, $v\in \mathrm{W}^{2,2}_u(S^1,TM)$, 
	\begin{align*}
	DE_{\sigma_n}(u)\cdot v&=\int_{0}^L\left(2\sigma_n^2 \scal{D_t\dot{u}}{D_t^2v}-(3\sigma^2\kappa(u)^2-1)\scal{\dot{u}}{D_tv}+2\sigma_n^2 \langle R(D_t\dot{u},\dot{u})\dot{u},v\rangle\right)d\leb^1\\
	&=DE_\sigma(u)\cdot v+(\sigma_n^2-\sigma^2)\int_{0}^L2\scal{D_t\dot{u}}{D_t^2v}-3\kappa(u)^2\scal{\dot{u}}{D_tv}+2\langle R(D_t\dot{u},\dot{u})\dot{u},v\rangle d\leb^1
	\end{align*}
	where $L=\mathfrak{L}(u)$ (recall that the arclength parametrization where $|\dot{u}|=1$ is possible because our Lagrangian is invariant under diffeomorphism). As a consequence, we have
	\begin{align*}
	\left|DE_{\sigma_n}(u)\cdot v-DE_\sigma(u)\cdot v\right|&\leq (\sigma_n^2-\sigma^2)\bigg\{2\np[{[0,L]}]{D_t\dot{u}}{2}\np[{[0,L]}]{D_t^2v}{2}+3\np[{[0,L]}]{D_tv}{\infty}\np[{[0,L]}]{D_t\dot{u}}{2}^2\\
	&\quad\quad\quad\quad\quad\quad\quad\quad\quad\quad\quad\quad\quad\quad\quad+2\np[M]{R}{\infty}\np[{[0,L]}]{D_t\dot{u}}{2}\np[{[0,L]}]{v}{\infty}\bigg\}\\
	&\leq C(u)(\sigma_n^2-\sigma^2)\Vert{v}\Vert_{\mathrm{W}_u^{2,2}(S^1,TM)}.
	\end{align*}
	Furthermore, the proof of theorem \ref{palaissmale} (where we prove Palais-Smale condition), shows the existence of a continuous function $f_{M}:\R_+^*\times\R_+\rightarrow\R_+$ increasing in each parameter, depending only on $(M^m,g)$ such that
	\begin{align*}
	C(u)\leq f_{M}(\sigma_n,E_{\sigma_n}(u))
	\end{align*}
	so for all $u$ satisfying \eqref{intervallePS},
	\begin{align}\label{borneunif}
	C(u)\leq f_{M}(\sigma_n,\beta(\sigma)+(\beta'(\sigma)+\epsilon)(\sigma_n-\sigma))
	\end{align}
	therefore $C(u)$ is uniformly bounded by a positive constant independent of $u$.
	We deduce that
	\begin{equation}\label{star}
	\sup\ens{\left|\left(DE_{\sigma_n}(u)-DE_{\sigma}(u)\right)\cdot v\right|, \Vert{v}\Vert_{\mathrm{W}_u^{2,2}(S^1,TM)}\leq 1}\underset{n\rightarrow\infty}{\longrightarrow}0.
	\end{equation}
	Now we can estimate the derivative of $t\mapsto E_\sigma(\phi_n^t(u))$ as follows :
	\begin{align}
	\frac{d}{dt}E_{\sigma}(\phi_n^t(u))&=DE_\sigma(\phi_n^t(u))\cdot X_n(\phi_n^t(u))\nonumber\\
	&\leq -\psi_n(\phi_n^t(u))DE_{\sigma_n}(\phi_n^t(u))\cdot X_n^0(\phi_n^t(u))+C(\phi_n^t(u))(\sigma_n^2-\sigma^2)\Vert{X_n^0(\phi_n^t(u))}\Vert_{\mathrm{W}^{2,2}_{\phi_n^t(u)}(S^1,TM)}\nonumber\\
	&\leq -2\psi_n(\phi_n^t(u))\delta+2C(\phi_n^t(u))(\sigma_n^2-\sigma^2)\label{décroissance}
	\end{align}
	
	For all $n\in\N$, let us a fix an element $A_n\in\mathscr{A}_0$ such that
	\begin{align*}
	\sup_{u\in A_n}E_{\sigma_n}(u)\leq \beta(\sigma_n)+\epsilon(\sigma_n-\sigma).
	\end{align*}
	For all $u\in A_n$, the map $t\mapsto E_{\sigma_n}(\phi_n^t(u))$ is decreasing, so for all $t\geq 0$, according to \eqref{borneunif},
	\begin{align*}
	E_{\sigma_n}(\phi_n^t(u))&\leq E_{\sigma_n}(u)\leq \beta(\sigma_n)+\epsilon(\sigma_n-\sigma)\\ C(\phi_n^t(u))&\leq f(\sigma_n,E_{\sigma_n}(\phi_n^t(u))\leq f(\sigma_0,\beta(\sigma)+\left(\beta'(\sigma)+2\epsilon\right)(\sigma_0-\sigma))
	\end{align*}
	By invariance of $\mathscr{A}_0$ under the action of the semi-flow $\ens{\phi_n^t}_{t\geq 0}$, for all $t\geq 0$, we define
	\[
	B_{A_n}(t)=\sup_{u\in A_n}E_{\sigma}(\phi_n^t(u))\geq \beta(\sigma)
	\]
	and $B_{A_n}(t)$ is attained only at points $u_n^t=\phi_n^t(u)$ satisfying \eqref{intervallePS}, and for such a $u_n^t$, we have
	\begin{align*}
	\beta(\sigma)-\epsilon(\sigma_n-\sigma)\leq E_\sigma(u_n^t)\leq E_{\sigma_n}(u_n^t)\leq \beta(\sigma)+(\beta'(\sigma)+2\epsilon)(\sigma_n-\sigma).
	\end{align*}
	Furthermore,
	\begin{align*}
	\partial_\sigma E_\sigma(u_n^t)\leq \beta'(\sigma)+3\epsilon,
	\end{align*}
	so if we choose $\epsilon={\left(8\sigma\log\dfrac{1}{\sigma}\right)}^{-1}$, as $\beta$ satisfies the entropy condition \eqref{entropie} at $\sigma$, we have	\begin{align*}
	\int_{S^1}\sigma^2\kappa^2(u_n^t)|\dot{u_n^t}|d\leb^1\leq \frac{7}{8}\frac{1}{\log\frac{1}{\sigma}},
	\end{align*}
	so for all $\sigma_n-\sigma\leq \dfrac{1}{8}$
	\begin{align*}
	\mathfrak{L}(u_n^t)&\geq \beta(\sigma)-\epsilon(\sigma_n-\sigma)-\int_{S^1}\sigma^2\kappa^2(u_n^t)|\dot{u_n^t}|d\leb^1\\
	&\geq \beta(0)-\frac{1}{\log\frac{1}{\sigma}}\\
	&\geq \frac{3}{4}\beta(0)
	\end{align*}
	for $\sigma\leq C(\beta(0))=e^{-\frac{4}{\beta(0)}}$. Therefore, for all $\sigma\leq C(\beta(0))$, and $n$ large enough such that $8(\sigma_n-\sigma)\leq 1$, we have $\psi_n(u_n^t)=1$, so thank of \eqref{décroissance}, if $n$ is large enough,
	\begin{equation}\label{dérivée}
	\frac{d}{dt}B_{A_n}(t)\leq -\delta
	\end{equation}
	thus
	\[
	B_{A_n}(t)\leq \beta(\sigma_n)+\epsilon(\sigma_n-\sigma)-\delta t
	\]
	so for $t$ large enough, $B_{A_n}(t)< \beta(\sigma)$, contradicting the definition of $\beta(\sigma)$.

	\textbf{Step 3}: convergence and conclusion.
	
	Thanks of step $2$, we can choose a sequence $\ens{u_n}_{n\in\N}\in\im$ satisfying \eqref{intervallePS}, and such that
	\[
	\lim_{n\rightarrow\infty}\Vert DE_{\sigma_n}(u_n)\Vert=0
	\]
	So \eqref{star} gives
	\[
	\sup_{n\in\N}E_{\sigma}(u_n)<\infty,\quad\text{and}\quad \Vert DE_{\sigma}(u_n)\Vert\underset{n\rightarrow\infty}{\longrightarrow} 0
	\]
	As a consequence, $\ens{u_n}_{n\in\N}$ is a Palais-Smale sequence for $E_\sigma$, \vspace{0.2em} we can suppose thanks of theorem \ref*{palaissmale} that there exists $u\in\im$ such that
	\[
	u_n\underset{n\rightarrow\infty}{\longrightarrow} u\quad\text{strongly in}\; \im
	\]
	In particular, thanks of \eqref{intervallePS}, we have
	\[
	\beta(\sigma)=\lim_{n\rightarrow\infty}E_{\sigma_n}(u_n)=E_{\sigma}(u)
	\]
	and
	\[
	\partial_\sigma E_\sigma(u)\leq\liminf_{n\rightarrow\infty}\partial_\sigma E_{\sigma_n}(u)\leq \beta'(\sigma)+3\epsilon,
	\]
	which concludes the proof of the proposition.
\end{proof}

We now come to the main result of this section.

\begin{theorem}\label{bonnextraction}
	There exists a sequence $\ens{\sigma_n}_{n\in\N}$ of positive numbers converging to $0$, and a sequence of critical points $\ens{u_n}_{n\in\N}$ of $\ens{E_{\sigma_n}}_{n\in\N}$ such that
	\begin{align*}
	\beta(0)\leq E_{\sigma_n}(u_n)=\beta(\sigma_n),\quad\text{et}\quad \lim_{n\rightarrow\infty}\sigma_n^2\int_{S^1}\kappa^2(u_n)|\dot{u}_n|d\leb^1=0.
	\end{align*}
\end{theorem}
\begin{proof}
	Choosing a sequence $\ens{\sigma_n}_{n\in\N}$ converging to $0$ such that for all $n\in\N$, the function $\beta$ satisfies the entropy condition \eqref{entropie} at $\sigma_n$ (which is possible as $\beta$ is differentiable $\leb^1$ almost everywhere and satisfies \eqref{limlog}), the theorem is now an easy consequence of the preceeding proposition.
\end{proof}

\section{Limiting Procedure}\label{Passage}

\begin{theorem}\label{géodésique}
Let $(M^m,g)$ a Riemannian compact manifold of class $C^\nu$ ($\nu\geq 3$), such that there exist an admissible subset $\A\subset \mathscr{P}^*(\mathrm{W}^{2,2}(S^1,M))$, and define $\A_0$ as. For all $\sigma\geq 0$, we define 
\begin{equation}
\beta(\sigma)=\inf_{A_0\in \mathscr{A}_0}\sup_{u\in A}E_\sigma(u).
\end{equation}
If for $\sigma$ small enough, $\beta(\sigma)<\infty$ and $\beta(0)>0$,
 there exists a sequence $\ens{\sigma_n}_{n\in\N}$ of positive numbers converging to $0$, verifying
 \begin{align*}
 E_{\sigma_n}(u_n)=\beta(\sigma_n), \quad\sigma_n^2\int_{S^1}\kappa^2(u_n)d\leb^1\leq \frac{1}{\log\frac{1}{\sigma_n}}
 \end{align*}
 and a closed non-trivial geodesic $u:S^1\rightarrow M$ such that
 $\ens{u_n}_{n\in\N}$ converges to $u$ strongly in $\mathrm{L}^{\infty}(S^1,M)$ and $\ens{\dot{u}_n}_{n\in\N}$ converge to $\dot{u}$ almost everywhere.
\end{theorem}
\begin{proof}
\textbf{Step 1}: quasi-conservation law and length convergence.

Let $\ens{u_n}_{n\in\N}$ a sequence given by the theorem \ref{bonnextraction}, in arc-length parametrization. We define, for all $n\in\N$, $L_n=\mathfrak{L}(u_n)$. Let $\ens{v_n}_{n\in\N}$ defined by
\[
v_n=\dot{u}_n-\sigma_n^2(2D_t^2\dot{u}_n+3\kappa^2(u_n)\dot{u}_n).
\]
\textit{A priori}, $v_n$ belongs to the dual of $\mathrm{W}^{2,2}_{u_n}(S^1,TM)$. However, thanks of \eqref{presque géodésique}, we have
\[
D_tv_n=2\sigma_n^2R(D_t\dot{u}_n,\dot{u}_n)\dot{u}_n\in \lp{2}{S^1}.
\]
Thus $v_n\in \mathrm{W}_{u_n}^{1,2}(S^1,TM)$, and
\begin{align}\label{scal}
\scal{\dot{u}_n}{v_n}&=1-\sigma_n^2\left(2\scal{D_t^2\dot{u}_n}{\dot{u}_n}+3\kappa^2(u_n)\right)\\
&=1-\sigma_n^2\kappa^2(u_n)
\end{align}
as $|\dot{u}_n|=1$, so $0=2\scal{D_t\dot{u}_n}{\dot{u}_n}$, and we have
\[
0=\scal{D_t^2\dot{u}_n}{\dot{u}_n}+\scal{D_t\dot{u}_n}{D_t\dot{u}_n}=\scal{D_t^2\dot{u}_n}{\dot{u}_n}+\kappa^2(u_n).
\]
Remark that $v_n\in \mathrm{W}^{1,2}([0,L_n])$ implies that $u_n$  is in $C^{\nu-1}([0,L_n],M)$. Indeed, 
\[
2\sigma_n^2 D_t^2\dot{u}_n=\dot{u}_n-v_n-3\sigma_n^2\kappa^2(u_n)\dot{u}_n\in \mathrm{L}^1([0,L_n])
\]
so $D_t^2\dot{u}_n\in L^1([0,L_n])$, and $u\in \mathrm{W}^{3,1}([0,L_n])$. An immediate bootstrap argument implies that $u_n\in \mathrm{W}^{\nu,1}(S^1,M)\subset C^{\nu-1}(S^1,M)$ and $v_n\in \mathrm{W}^{\nu-1,1}(S^1,M)\subset C^{\nu-2}(S^1,M)$.
Furthermore, $D_tv_n\underset{n\rightarrow\infty}{\longrightarrow}0$ in $\lp{2}{S^1}$. Indeed (recall $L_n=\mathfrak{L}(u_n)$),
\begin{align*}
\int_{0}^{L_n}\sigma_n^2 |R(D_t\dot{u}_n,\dot{u}_n)\dot{u}_n|^2d\leb^1
\leq\np[M]{R}{\infty}\int_{0}^{L_n}\sigma_n^2|D_t\dot{u}_n|^2d\leb^1=\np[M]{R}{\infty}\int_{S^1}\kappa^2(u_n)|\dot{u}_n|d\leb^1,
\end{align*}
We deduce that
\[
\np{D_tv_n}{2}\leq 2\np[M]{R}{\infty}\sigma_n{\left(\sigma_n^2\int_{S^1}\kappa^2(u_n)|\dot{u}_n|d\leb^1\right)}^{\frac{1}{2}}\underset{n\rightarrow\infty}{\longrightarrow}0
\]
In particular, there exist $\bar{v}\in M$ such that
\[
\np{v_n-\bar{v}}{\infty}\underset{n\rightarrow\infty}{\longrightarrow} 0.
\]
If we set $\displaystyle\bar{v_n}=\frac{1}{2\pi}\int_{S^1}v_nd\leb^1$, this is equivalent to
\[
\np{v_n-\bar{v_n}}{\infty}\underset{n\rightarrow\infty}{\longrightarrow} 0,
\]
which in turn implies that
\begin{align*}
\int_{0}^{L_n}v_n\cdot \bar{v_n}d\leb^1=\int_{0}^{L_n}|\bar{v_n}|^2d\leb^1+\int_{0}^{L_n}(v_n-\bar{v_n})\cdot \bar{v_n}d\leb^1\\
=L_n|\bar{v_n}|^2+\int_{0}^{L_n}(v_n-\bar{v_n})\cdot \bar{v_n}d\leb^1
\end{align*}
and
\[
\epsilon_n=\int_{0}^{L_n}(v_n-\bar{v_n})\cdot \bar{v_n}d\leb^1\leq L_n|\bar{v_n}|\np{v_n-\bar{v_n}}{\infty}\underset{n\rightarrow\infty}{\longrightarrow}0.
\]
On the other hand,
\[
\int_{0}^{L_n}v_n\cdot \bar{v_n}d\leb^1=\int_{0}^{L_n}\dot{u}_n\cdot \bar{v_n}d\leb^1-3\sigma_n^2\int_{0}^{L_n}\kappa^2(u_n)\dot{u_n}\cdot \bar{v_n}d\leb^1\\
\leq L_n|v_n|+3\sigma_n^2\int_{0}^{L_n}\kappa^2(u_n)d\leb^1|v_n|
\]
so
\[
L_n|\bar{v_n}|^2+\epsilon_n\leq L_n|\bar{v_n}|+3\sigma_n^2\int_{S^1}\kappa^2(u_n)|\dot{u}_n|d\leb^1.
\]
Now, $L_n\geq \beta(0)$ pour tout $n\in\N$, so
\[
|\bar{v_n}|\leq 1+\frac{3}{\beta(0)}\sigma_n^2\int_{S^1}\kappa^2(u_n)|\dot{u}_n|d\leb^1\underset{n\rightarrow\infty}{\longrightarrow} 1.
\]
And we get
\begin{equation}\label{normev}
|\bar{v}|=\lim_{n\rightarrow\infty}\frac{1}{L_n}\int_{0}^{L_n}|v_n|d\leb^1\leq 1.
\end{equation}

\textbf{Step 2} : weak convergence

The sequence $\ens{u_n}_{n\in\N}$ is bounded in $\mathrm{W}^{1,\infty}(S^1,M)$, as $(M,g)$ is compact, and $\ens{L_n}_{n\in\N}$ is bounded. Therefore, Arzelà-Ascoli and Banach-Alaoglu, imply that we can extract a subsequence from $\ens{u_n}_{n\in\N}$ (which is still denoted $\ens{u_n}_{n\in\N}$), such that $\ens{u_n}_{n\in\N}$ in $\mathrm{L}^{\infty}$ and weakly-* to a function $u\in\mathrm{W}^{1,\infty}([0,L],M)$. In particular $\ens{\dot{u}_n}_{n\in\N}$ converges almost everywhere to $\dot{u}$, for all interval $I$ such that for $n$ large enough, $I\subset [0,L_n]$, we have
\begin{align*}
\int_{I}\scal{\dot{u}_n}{\bar{v}_n}d\leb^1\conv \int_{I}\scal{\dot{u}}{\bar{v}}d\leb^1
\end{align*}
and according to \eqref{scal}, 
\[
\frac{1}{\leb^1(I)}\int_{I}\scal{\dot{u}_n}{\bar{v}_n}=1-\frac{1}{\leb^1(I)}\int_{I}\sigma_n^2\kappa^2(u_n)d\leb^1\conv 1.
\]
Furthermore, as $|\dot{u}_n|=1$, and $\ens{\dot{u}_n}_{n\in\N}$ converges almost everywhere to $\dot{u}$ so $|\dot{u}|\leq 1$. According to \eqref{normev}, $|\bar{v}|\leq 1$, so thanks of Cauchy-Schwarz inequality, we have $|\dot{u}|=1$, and $|\bar{v}|=1$. We deduce that
\[
L=\int_{S^1}|\dot{u}|d\leb^1=\lim_{n\rightarrow\infty}L_n\geq \beta(0),
\]
As $E_{\sigma_n}(u_n)=\beta(\sigma_n)$, and $\beta(\sigma_n)\conv \beta(0)$, we get $L=\beta(0)$. Indeed,
\begin{align*}
\beta(0)=\beta(\sigma_n)+o(1)=L_n+\int_{S^1}\sigma_n^2\kappa^2(u_n)|\dot{u}_n|d\leb^1+o(1),
\end{align*}
so $L_n\conv \beta(0)$.

\textbf{Step 3}: limiting equation.

We wish now to pass to the limit in the Euler-Lagrange equation. We need the following technical lemma, stated separetely for the sake of clarity.
\begin{lemme}\label{lemmealacon}
Let $v\in \wt$, and for all $n\in\N$, $v_n=P(u_n)v\in \wtn$, where $P(u_n)$ is the orthogonal projection on $T_{u_n}M$. We have
\begin{align*}
1)\quad&v_n\overset{\mathrm{L}^\infty}\conv v,\\
2)\quad&\ens{D_tv_n}_{n\in\N}\;\text{is bounded in}\;\mathrm{L}^\infty\;\text{et}\; D_tv_n\overset{\mathrm{L}^2}\conv D_tv,\\
3)\quad&\int_{0}^{L_n}\sigma_n^2|D_t^2v_n|d\leb^1\conv 0.
\end{align*}
\end{lemme}
\begin{proof}[Proof of lemma \eqref{lemmealacon}]
If $\mathfrak{n}$ is a normal vector field to $M$, the orthogonal projection $P(u_n):\R^q\rightarrow T_{u_n}M$ is given by
\begin{align*}
P(u_n)v=v-\mathfrak{n}(u_n)\scal{\mathfrak{n}(u_n)}{v}=v-\mathfrak{n}(u_n)\scal{\mathfrak{n}(u_n)-\mathfrak{n}(u)}{v}
\end{align*}
and $u_n\conv u$ in $\mathrm{L}^\infty$, so $P(u_n)v\conv v$ in $L^{\infty}$.
\begin{align*}
D_t(P(u_n)v-v)&=-D\mathfrak{n}(u_n)[\dot{u}_n]\scal{\mathfrak{n}(u_n)-\mathfrak{n}(u)}{v}
-\mathfrak{n}(u_n)\scal{D\mathfrak{n}(u_n)[\dot{u}_n]
-D\mathfrak{n}(u)[\dot{u}]}{v}\\&-\mathfrak{n}(u_n)\scal{\mathfrak{n}(u_n)-\mathfrak{n}(u)}{D_tv}.
\end{align*}
Therefore, $D_t(P(u_n)v)$ is bounded in $\mathrm{L}^{\infty}$, and converges into $\mathrm{L}^2$ to $D_tv$. Finally, 
\begin{align*}
D_t^2(P(u_n)v)=D^2P(u_n)[\dot{u}_n,\dot{u}_n]v+DP(u_n)[D_t\dot{u}_n]v
+2P(u_n)[\dot{u}_n]D_tv+P(u_n)D_t^2v
\end{align*}
and $P$ is $C^{\nu-1}$ (and $\nu-1\geq 2$), so there exist a constant $C$ independent of $n$ such that
\begin{align*}
\int_{0}^{L_n}\sigma_n^2|D_t^2 (P(u_n)v)|^2d\leb^1&\leq
C\sigma_n^2\nwp{v}{2,2}{S^1}+C\sigma_n\np{v}{2}{\left(\int_{0}^{L_n}\sigma_n^2\kappa^2(u_n)d\leb^1\right)}^{\frac{1}{2}}
\conv 0.
\end{align*}
which completes the proof of the lemma.
\end{proof}
For all $\sigma>0$, define $F_\sigma=DE_\sigma-DE_0$. We have
\begin{align*}
\sigma_n^2F_{\sigma_n}(u_n)\cdot v_n&=\sigma_n^2\int_{0}^{L_n}2\scal{D_t^2v_n}{D_t\dot{u}_n}+3\kappa^2(u_n)\scal{\dot{u}_n}{D_tv_n}+2\scal{R(D_t\dot{u}_n,\dot{u}_n)\dot{u}_n}{v_n}d\leb^1\\
\begin{split}
\leq2{\left(\int_{0}^{L_n}\sigma_n^2|D_t^2v_n|d\leb^1\right)}^{\frac{1}{2}}{\left(\int_{0}^{L_n}\sigma_n^2|D_t\dot{u}_n|d\leb^1\right)}^{\frac{1}{2}}+3\np{D_tv_n}{\infty}\int_{0}^{L_n}\sigma_n^2\kappa^2(u_n)d\leb^1\\+2\np[M]{R}{\infty}\np{v_n}{2}\sigma_n{\left(\int_{0}^{L_n}\sigma_n^2\kappa^2(u_n)d\leb^1\right)}^{\frac{1}{2}}\conv 0
\end{split}
\end{align*}
whereas
\begin{align*}
\int_{0}^{L_n}\scal{D_tv_n}{\dot{u}_n}d\leb^1\conv\int_{0}^{L}\scal{D_tv}{\dot{u}}d\leb^1.
\end{align*}
As for all $n\in\N$, $u_n$ is a critical point of $E_{\sigma_n}$, we have
\begin{align*}
DE_{\sigma_n}(u_n)=\int_{0}^{L_n}\scal{D_tv_n}{\dot{u}_n}d\leb^1+\sigma_n^2F_{\sigma_n}(u_n)\cdot v_n=0
\end{align*}
so we deduce
\begin{align*}
\int_{0}^{L}\scal{D_tv}{\dot{u}}d\leb^1=0
\end{align*}
for all $v\in\wt$, \textit{i.e.} $u$ is a distributional solution of
\begin{equation}\label{géodésiquesfinale}
\frac{d^2}{dt^2}u+\fond(\dot{u},\dot{u})=0
\end{equation}
where $\fond$ is the second fundamental form of the immersion $u:S^1\rightarrow (M^m,g)$.
This implies that $\dfrac{d^2}{dt^2}u\in L^{\infty}([0,L],M)$, so $u\in\mathrm{W}^{2,\infty}([0,L])$, and by a immediate bootstrap argument, we get that actually $u\in C^{\nu}([0,L])$, $|\dot{u}|=1$,
\begin{equation}
D_tu=\frac{d^2}{dt^2}u+\fond(\dot{u},\dot{u})=0
\end{equation}
We conclude that $u$ is a non-trivial closed geodesic of length $\beta
(0)>0$.
\end{proof}

\section{Admissible Family Construction}\label{construction}

\begin{theorem}\label{admissiblefamily}
Let $(M^m,g)$ a compact Riemannian manifold of class $C^\nu$ ($\nu\geq 3$). There exists an admissible set $\A$ in $\mathrm{W}^{2,2}(S^1,M)$, in the sense of definition \ref{admissible}.
\end{theorem}
\begin{proof}
According to Hurewicz theorem, if $M$ is a compact manifold, and $\pi_1(M)={1}$ (otherwise, we can minimize directly on a non-trivial homotopy class), then if $k$ is an integer such that $H_i(M)=\ens{1}$ for all $i<k$, $H_k(M)\neq \ens{1}$, then $\pi_i(M)=\ens{1}$ for all $i<k$, and $\pi_k(M)\simeq H_k(M)$. As $H_m(M^m,\Z)\simeq \Z$, there exist $k\leq m$ such that $\pi_k(M)\neq\ens{1}$. Let $f:S^k\rightarrow M$ a continuous homotopically non-trivial. We may assume that $f$ is of class $C^\nu$, because according to Whitney theorem, every continuous map between manifolds is homotopic to a regular map (see \cite{hirsch}, \cite{whitney}). On $S^k$ let us consider the following classical sweepout
\begin{align}\label{sweepout}
\ens{x_3=1-2t_3,\cdots,x_{k+1}=1-2t_{k+1}}
\end{align}
où $t_3,\cdots, t_{k+1}\in [0,1]$. This gives a map $g:S^k\rightarrow S^k$ of degree $1$. Write, for $t\in [0,1]^{k-1}$,  $u_t:S^1\rightarrow S^k$ the circle defined by \eqref{sweepout}. Then for all but finitely $t\in [0,1]^{k-1}$, $u_t$ is an immersed curve.
We define
\begin{align*}
\A=\ens{\ens{\phi\circ f\circ u_t}_{t\in [0,1]^{k-1}}, \phi\in \mathrm{Homeo}_{\,0}(\im)}
\end{align*}
where $\mathrm{Homeo}_{\,0}(\im)$ is the set of locally Lipschitz homeomorphisms of $\im$ isotopic to the identity map.
The manifold $(M^m,g)$ being compact, its injectivity radius $\mathrm{inj}(M)$ is positive. Let us show that for all $\phi\in\mathrm{Homeo}_{\,0}(\im)$, 
\begin{align}\label{injectivite}
\sup_{t\in[0,1]^{k-1}}\mathfrak{L}(\phi\circ f\circ u_t)\geq \mathrm{inj}(M).
\end{align}
We argue by contradiction. If if \eqref{injectivite} is not satisfied, for all $t\in [0,1]^{k-1}$, the curve $\phi\circ f\circ u_t:S^1\rightarrow M$ is null-homotopic. This implies that $\phi\circ f\circ g$ is null-homotopic. Now $\phi$ is isotopic to the identity map, so $\phi\circ f\circ g$ is homotopic to $f\circ g$. But as $g:S^k\rightarrow S^k$ is a degree one map, and $f:S^k\rightarrow M$ homotopically non-trivial, $f\circ g:S^k\rightarrow M$ cannot be null-homotopic. We deduce that
\begin{align*}
\beta(0)=\inf_{A\in \A}\sup_{u\in\A}\mathfrak{L}(u)\geq \mathrm{inj}(M)>0, 
\end{align*}
which concludes the proof of the theorem.
\end{proof}

\section{Lower Semi-Continuity of the Index}\label{indice}

Motivating by the construction by a min-max viscosity method of minimal surfaces of given index, we aim at proving here that the index of the constructed curves in lower semi-continuous.

\begin{defi}
Let $\sigma\geq 0$, and $u$ a critical point of $E_\sigma$. The index of $u$, noted $\mathrm{Ind}(u)\in\N\cup\ens{\infty}$, is equal to the dimension of the larger subspace of $\mathrm{W}_u^{2,2}(S^1,TM)$, on which the second derivative $D^2E_\sigma(u)$ (defined by \eqref{hessienne}) is negative semi-definite.
\end{defi}

The proof of the index lower semi-continuity will be an easy consequence of the following lemma.
\begin{lemme}\label{lemmeindice}
Let $\ens{\sigma_n}_{n\in\N}$ a sequence of positive real numbers converging to $0$. If $\ens{u_n}_{n\in\N}$ is a sequence of critical points associated to $\ens{E_{\sigma_n}}_{n\in\N}$, such that $\ens{u_n}_{n\in\N}$ (resp. $\ens{\dot{u}_n}_{n\in\N}$) converge in $L^{\infty}$ (resp. almost everywhere) to a closed non-trivial geodesic $u\in \im$ (resp. to $\dot{u}$) of length $L>0$. If $\ens{v_n}_{n\in\N}$ is a sequence verifying $v_n\in\mathrm{W}_{u_n}^{2,2}(S^1,TM)$, $v_n\overset{\mathrm{L}^{\infty}}{\conv} v\in \mathrm{W}_u^{2,2}(S^1,TM)$, $D_tv_n\overset{\mathrm{L}^2}\conv D_tv$,  $\ens{v_n}_{n\in\N}$ is bounded in $\mathrm{L}^\infty$, and
\begin{align*}
\int_{S^1}\sigma_n^2\kappa^2(u_n)|\dot{u}_n|d\leb^1\conv 0, \quad \int_{S^1}\sigma_n^2|D_t^2v_n|^2|\dot{u}_n|d\leb^1\conv 0.
\end{align*}
Then
\begin{align*}
D^2E_{\sigma_n}(u_n)[v_n,v_n]\conv D^2E_0(u)[v,v]=\int_{0}^L\left(|D_tv|^2-\scal{D_tv}{\dot{u}}^2-\scal{R(\dot{u},v)v}{\dot{u}}\right)d\leb^1.
\end{align*}
\end{lemme}

\begin{proof}
The hypothesis implies that $\ens{v_n}_{n\in\N}$ is bounded in $L^{\infty}(S^1,M)$, and in $\mathrm{W}^{1,2}(S^1,M)$. Furthermore, theorem \ref{géodésique} shows that $\ens{\dot{u_n}}_{n\in\N}$ converges to $L^p(S^1,M)$ (if $L_n=\mathfrak{L}(u)$, we have $L_n\conv L$) for all $1\leq p<\infty$ according to Lebesgue's dominated convergence theorem.

All estimates are elementary, using only Cauchy-Schwarz inequality, and otherwise, $R$ being a $C^{\nu-2}$ ($\nu-2\geq 1$) tensor on the compact $C^\nu$ manifold $(M,g)$ implies that
\begin{align*}
\max\{\np[M]{R}{\infty},\np[M]{\D R}{\infty}\}<\infty
\end{align*}

As for all $n\in\N$,  $|\dot{u}_n|=1$, 
\begin{align*}
&\sup_{n\in\N}\np[M]{\nabla_{\dot{u}_n}R}{\infty}\leq \np[M]{\D R}{\infty}<\infty\\
&\sup_{n\in\N}\np[M]{\nabla_{v_n}R}{\infty}\leq \np[M]{\D R}{\infty}\sup_{n\in\N}\np{v_n}{\infty}<\infty
\end{align*}
We have
\begin{align*}
\sigma_n^2\int_{0}^{L_n}|D_t^2v_n|^2+|R(v_n,\dot{u}_n)\dot{u}_n|^2d\leb^1\leq\int_{0}^{L_n}\sigma_n^2|D_t^2v_n|^2d\leb^1+\sigma_n^2\np[M]{R}{\infty}\npl{v_n}{2}^2\conv 0
\end{align*}
We write
\begin{align*}
K_n=\left(\int_{S^1}\sigma_n^2\kappa^2(u_n)|\dot{u}_n|d\leb^1\right)^{\frac{1}{2}}=\left(\int_{0}^{L_n}\sigma_n^2\kappa^2(u_n)d\leb^1\right)^{\frac{1}{2}}\conv 0.
\end{align*}
We estimate the other terms as following.
\begin{align*}
&\sigma_n^2\left|\int_0^{L_n}\left(4\scal{D_tv_n}{\dot{u}_n}^2+2\scal{D_tv_n}{\dot{u}_n}-|D_tv_n|^2+\scal{R(\dot{u}_n,v_n)v_n}{\dot{u}_n}\right)\kappa^2(u_n)d\leb^1\right|\\
&\leq \left(4\npl{D_tv_n}{\infty}^2+2\npl{D_tv_n}{\infty}+\np[M]{R}{\infty}\npl{v_n}{\infty}\right)K_n^2
\end{align*}
\begin{multline*}
\sigma_n^2\left|\int_{0}^{L_n}\left(\scal{D_t^2v_n}{\dot{u}_n}+\scal{D_tv_n}{D_t\dot{u}_n}\right)^2d\leb^1\right|\leq 2\int_{0}^{L_n}\sigma_n^2\npl{D_t^2v_n}{2}^2+2\npl{D_tv_n}{\infty}^2K_n^2
\end{multline*}
\begin{align*}
\sigma_n^2\left|\int_0^{L_n}\scal{\D_{\dot{u}_n}R(v_n,\dot{u}_n)v_n}{D_t\dot{u}_n}d\leb^1\right|&\leq \sigma_n\np[M]{\D R}{\infty}\npl{v_n}{\infty}\npl{v_n}{2}K_n
\end{align*}
\begin{align*}
\sigma_n^2\left|\int_{0}^{L_n}\scal{\D_{v_n}R(v_n,\dot{u}_n)\dot{u}_n}{D_t\dot{u}_n}d\leb^1\right|\leq \sigma_n\np[M]{R}{\infty}\npl{v_n}{\infty}\npl{v_n}{2}K_n
\end{align*}
\begin{align*}
\sigma_n^2\left|\int_{0}^{L_n}\scal{R(D_tv_n,\dot{u}_n)v_n}{D_t\dot{u}_n}d\leb^1\right|\leq \sigma_n\np[M]{\D R}{\infty}\npl{v_n}{\infty}\npl{D_tv_n}{2}K_n
\end{align*}
\begin{align*}
\sigma_n^2\left|\int_{0}^{L_n}\scal{R(D_t\dot{u}_n,v_n)v_n}{D_t\dot{u}_n}d\leb^1\right|\leq\np[M]{R}{\infty}\npl{v_n}{\infty}^2K_n^2
\end{align*}
\begin{align*}
\sigma_n^2\left|\int_{0}^{L_n}\scal{R(v_n,\dot{u}_n)D_tv_n}{D_t\dot{u}_n}d\leb^1\right|\leq\sigma_n\np[M]{R}{\infty}\npl{v_n}{\infty}\npl{D_tv_n}{2}K_n
\end{align*}
\begin{align*}
\sigma_n^2\left|\int_{0}^{L_n}\scal{R(\dot{u}_n,D_tv_n)\dot{u}_n}{D_t\dot{u}_n}d\leb^1\right|\leq\sigma_n\np[M]{R}{\infty}\npl{D_tv_n}{2}K_n
\end{align*}
\begin{align*}
\sigma_n^2\left|\int_{0}^{L_n}\scal{R(v_n,\dot{u}_n)D_tv_n}{D_t\dot{u}_n}d\leb^1\right|\leq\sigma_n\np[M]{R}{\infty}\npl{v_n}{\infty}\npl{D_tv_n}{2}K_n
\end{align*}
\begin{align*}
\sigma_n^2\left|\int_{0}^{L_n}\scal{D_tv_n}{\dot{u}_n}\scal{R(v_n,\dot{u}_n)\dot{u}_n}{D_t\dot{u}_n} d\leb^1\right|\leq\sigma_n\np[M]{R}{\infty}\npl{v_n}{\infty}\npl{D_tv_n}{2}K_n
\end{align*}
\begin{align*}
&\sigma_n^2\left|\int_0^{L_n}\scal{D_tv_n}{\dot{u}_n}\left(\scal{D_t^2v_n}{D_t\dot{u}_n}-2\scal{D_tv_n}{\dot{u}_n}\kappa^2(u_n)+\scal{R(v_n,\dot{u}_n)\dot{u}_n}{D_t\dot{u}_n}\right)d\leb^1\right|\\
&\leq \sigma_n^2\npl{D_tv_n}{\infty}\np{D_t^2v_n}{2}^2K_n+2\npl{D_tv_n}{\infty}^2K_n^2\\
&\quad+\sigma_n\np[M]{R}{\infty}\npl{v_n}{\infty}\npl{D_tv_n}{2}K_n
\end{align*}
\begin{align*}
&\sigma_n^2\left|\int_0^{L_n}\left(|D_tv_n|^2-\scal{D_tv_n}{\dot{u}_n}^2-\scal{R(\dot{u}_n,v_n)v_n}{\dot{u}_n}\right)\kappa^2(u_n)d\leb^1\right|\\
&\leq \left(\npl{D_tv_n}{\infty}^2+\np[M]{R}{\infty}\npl{v_n}{\infty}^2\right)K_n^2.
\end{align*}
Finally, the unit vector sequence $\ens{\dot{u}_n}_{n\in\N}$ converge almost everywhere to $\dot{u}$ (which is also a unit vector), so we can apply Lebesgue's dominated convergence theorem to get
\begin{align*}
&\int_{S^1}\left(|\D_{\frac{\dot{u}_n}{|\dot{u}_n|}}v_n|^2-\scal{\D_{\frac{\dot{u}_n}{|\dot{u}_n|}}v_n}{\frac{\dot{u}_n}{|\dot{u}_n|}}^2-\scal{R\left(\frac{\dot{u}_n}{|\dot{u}_n|},v_n\right)v_n}{\frac{\dot{u}_n}{|\dot{u}_n|}}\right)|\dot{u_n}|d\leb^1\\
&\conv \int_{S^1}\left(|\D_{\frac{\dot{u}}{|\dot{u}|}}v|^2-\scal{\D_{\frac{\dot{u}}{|\dot{u}|}}v}{\dot{u}}^2-\scal{R\left(\frac{\dot{u}}{|\dot{u}|},v\right)v}{\frac{\dot{u}}{|\dot{u}|}}\right)|\dot{u}|d\leb^1,
\end{align*}
which completes the proof of the lemma.
\end{proof}

\begin{theorem}\label{indextheorem}
Under the hypothesis of \ref{géodésique}, if $\ens{u_n}_{n\in\N}$ is the sequence of critical points associated to $\ens{E_{\sigma_n}}$, to a non-trivial closed geodesic $u\in\im$, we have
\begin{align*}
\mathrm{Ind}(u)\leq \liminf_{n\rightarrow\infty}\mathrm{Ind}(u_n).
\end{align*}
\end{theorem}
\begin{proof}
If $v\in\mathrm{W}_{u}^{2,2}(S^1,TM)$, and $P(u_n)$ is the orthogonal projection $\R^q\rightarrow T_{u_n}M$, if $v_n=P(u_n)$, la suite $\ens{v_n}_{n\in\N}$ thanks of lemma \ref{lemmealacon}, $\ens{v_n}_{n\in\N}$ satisfies the hypothesis of lemma \ref{lemmeindice}. If $v^1,\cdots,v^{I}\in\mathrm{W}^{1,2}(S^1,M)$ is a free orthonormal family in $\mathrm{L}^2(S^1,M)$ such that $D^2E_0$ is negative semi-definite on $\mathrm{Span}\ens{v^1,\cdots,v^I}$, then, if $v_n^j=P(u_n)v^j$, the family $\ens{v_n^1,\cdots,v_n^I}$ is free in $\mathrm{W}_{u_n}^{2,2}(S^1,M)$, for $n$ large enough. As
$D^2E_0(u)[v^j,v^j]<0$, we have
\[
D^2E_{\sigma_n}(u_n)[v_n,v_n]\conv D^2E_0(u)[v^j,v^j]<0
\] 
so for $n$ large enough, $D^2E_{\sigma_n}(u_n)[v_n,v_n]<0$, and $\ens{v_n^1,\cdots,v_n^I}$ is free, thus
\begin{align*}
I\leq \liminf_{n\rightarrow\infty}\mathrm{Ind}(u_n).
\end{align*}
This implies that
\begin{align*}
\mathrm{Ind}(u)\leq \liminf_{n\rightarrow\infty}\mathrm{Ind}(u_n).
\end{align*}
which concludes the proof of the theorem.
\end{proof}

\newpage

\section{Counter-examples}\label{counterexamples}

\subsection{Counter-examples in Dimension 1}

Let $(M^2,g)$ a compact $C^3$ Riemannian surface of constant Gauss curvature $K_M\in\R$ (which is just equal to the sectional curvature in our convention). Let $\sigma>0$, and $u_\sigma$ a critical point of $E_\sigma$. We know that $u$ is $C^2$ and satisfies \eqref{presque2}
\begin{equation}\label{presque2}
D_t\dot{u}=\sigma^2\left\{D_t\left(2D^2_t\dot{u}+3\kappa^2\dot{u}\right)+2R(D_t\dot{u},\dot{u})\dot{u}\right\}
\end{equation}

Let $\nu$ a normal vector to the curve $u$, et $k$ the signed curvature, defined as 
\[
D_t\dot{u}=k\nu
\]
As $\scal{D_t\dot{u}}{\dot{u}}=0$, $k$ is well-defined. Moreover, Frénet equations in dimension $2$ imply that
\[
D_t\nu=-k\dot{u},
\]
so
\begin{align*}
 D_t^2\dot{u}&=\dot{k}\nu-k^2\dot{u},\\
 D_t^3\dot{u}&=\ddot{k}\nu+\dot{k}(-k\dot{u})-2\dot{k}k\dot{u}-k^2(k\nu)=\ddot{k}\nu-3\dot{k}k\dot{u}-k^3\nu.
\end{align*}
As 
\begin{align*}
D_t(k^2\dot{u})=2\dot{k}k\dot{u}+k^3\nu
\end{align*}
the equation \eqref{presque2} is equivalent to
\[
k\nu=\sigma^2(2\ddot{k}+k^3)\nu+2\sigma^2kR(\nu,\dot{u})\dot{u}
\]
Taking the scalar product with $\nu$, we get
\begin{equation}\label{bonneq}
k_\sigma(t)=\sigma^2(2\ddot{k}_\sigma(t)+k_\sigma^3(t)+2K_Mk_\sigma(t))
\end{equation}

We can explicitly solve this equation thanks of Jacobi elliptic functions (see \cite{LZ1}, \cite{Handbook}).

Let $0\leq p<1$, we define $f_p:\R\rightarrow\R$
\begin{equation}
f_p(t)=\int_{0}^t\frac{d\theta}{\sqrt{1-p^2\sin^2\theta}}
\end{equation}
and Jacobi elliptic functions $\sn,\cn$ and $\dn$ as
\begin{align*}
\sn(t,p)&=\sin f_p^{-1}(t)\\
\cn(t,p)&=\cos f_p^{-1}(t)\\
\dn(t,p)&=\sqrt{1-p^2\sn^2(t)}
\end{align*}
and the function $\K:[0,1[\rightarrow \R_+$ by
\[
\K(p)=f_p\left(\frac{\pi}{2}\right)=\int_0^{\frac{\pi}{2}}\frac{d\theta}{\sqrt{1-p^2\sin^2\theta}}.
\]
The functions $\sn, \cn$, are $4\K$-periodic, and $\dn$ is $2\K$-periodic. If we write $\sn_p=\sn(\cdot,p)$, $\cn_p=\cn(\cdot,p)$, $\dn_p=\dn(\cdot,p)$, we have
\begin{align*}
\dot{\sn_p}&=\cn_p\dn_p\\
\dot{\cn_p}&=-\cn_p\dn_p\\
\dot{\dn_p}&=-p^2\sn_p\cn_p
\end{align*}
and if $\dn=\dn(\cdot,p)$ ($0\leq p<1$),
\begin{align*}
\ddot{\dn}+2\dn^3-(2-p^2)\dn=0.
\end{align*}
The function $t\mapsto u(t)=a\,\dn(bt,p)$, $u$ is a solution of the differential equation
\[
\ddot{u}+2\left(\frac{b}{a}\right)^2u^3-b^2(2-p^2)u=0
\]
Fix $0\leq p<1$, we have
\begin{align}\label{formulecounter}
k_\sigma(t)=\pm \left(\frac{1}{\sigma^2}\frac{2(1-2\sigma^2K_M)}{(2-p^2)}\right)^\frac{1}{2}\dn\left({\left(\frac{1-2\sigma^2K_M}{2(2-p^2)}\right)}^{\frac{1}{2}}\frac{t}{\sigma},p\right)
\end{align}
and
\[
\sigma^2k_\sigma^2(t)= 2\frac{1-2\sigma^2K_M}{2-p^2}\left(1-p^2\sn^2\left(\left(\frac{1-2\sigma^2K_M}{2(2-p^2)}\right)^{\frac{1}{2}}\frac{t}{\sigma},p\right)\right)
\]
If $C(\sigma)$ is the constant
\[
C(\sigma)=\left(\frac{1-2\sigma^2K_M}{2(2-p(\sigma)^2)}\right)^{\frac{1}{2}}
\]
Then, $k_\sigma^2$ is a $2\sigma C(\sigma)^{-1}\K(p(\sigma))$-periodic function and $L(\sigma)$-periodic ($L(\sigma)=\mathfrak{L}(u)$). Thus there exists $m(\sigma)\in\N$ such that $L(\sigma)=2\sigma m(\sigma)C(\sigma)^{-1}\K(p(\sigma))$. In particular (see \cite{LZ1}, p. 19),
\begin{align*}
\int_{0}^{L(\sigma)}\sigma^2k_\sigma^2(t)dt&=8\sigma\, m(\sigma)C(\sigma)\int_{0}^{\K(\sigma)}\dn_{p(\sigma)}^2(t)dt\\
&=4\sigma\,m(\sigma)\sqrt{\frac{2-4\sigma^2K_M}{2-p(\sigma)^2}}\int_{0}^{\frac{\pi}{2}}\sqrt{1-p(\sigma)^2\sin^2(t)}dt
\end{align*}

If $\ens{\sigma_n}_{n\in\N}$ is a sequence of positive numbers converging to $0$ such that
\begin{align*}
L=\lim_{n\rightarrow\infty} L(\sigma_n)>0,
\end{align*}
if $p(\sigma_n)\conv p\in [0,1[$, we have $\displaystyle\K(p(\sigma_n))\conv \K(p)\in \left[\frac{\pi}{2},\infty\right)$
and
\begin{align}
\int_{0}^{L(\sigma_n)}\sigma_n^2 k_{\sigma_n}^2(t)dt &= \frac{2L(\sigma_n)C(\sigma_n)^2}{\K(p(\sigma_n))}\int_{0}^{\frac{\pi}{2}}\sqrt{1-p(\sigma_n)^2\sin^2(t)}dt\nonumber\\
&\conv \frac{4L}{(2-p^2)\K(p)}\int_{0}^{\frac{\pi}{2}}\sqrt{1-p^2\sin^2(t)}\,dt>0,\label{degenerescence}
\end{align}
which give a family of counter-examples, as we will see in next section.

\subsubsection{Explicit Counter-example on $S^2$}

The goal of this section is to prove the following result.

\begin{prop}\label{countersphere}
	On $S^2$ equipped with its standard metric, let $\A$ the admissible set of curves given by the canonical sweep-out on $S^2$.  There exists a sequence $\ens{\sigma_n}_{n\in\N}$ of positive real numbers converging to $0$ and a sequence of critical points $\ens{u_n}_{n\in\N}$ of $\ens{E_{\sigma_n}}_{n\in\N}$, and a curve $u\in \mathrm{W}^{1,2}(S^1,M)$, such that \vspace{-0.4em}
	\begin{align*}
	E_{\sigma_n}(u_n)\conv \beta(0)=\pi, \quad \mathfrak{L}(u_n)\conv \frac{\pi}{2}
	\end{align*}
	and
	\begin{align*}
	u_n\conv[\mathrm{L}^\infty] u\quad\text{strongly},\quad \quad u_n\xrightarrow[n\rightarrow\infty]{\mathrm{W}^{1,2}} u\quad\text{weakly},\quad \text{and}\quad \dot{u}_n\centernot{\xrightarrow[n\rightarrow\infty]{}} \dot{u} \quad \text{a.e.}
	\end{align*}
	Furthermore, there exists a negligible subset $N\subset S^1$ such that
	$\ens{\dot{u}_n(t)}_{n\in\N}$ has no limit point for all $t\in S^1\setminus N$,
	and for all open interval $I\subset S^1$,
	\begin{align*}
	\mathfrak{L}(u|I)<\liminf_{n\rightarrow\infty}\mathfrak{L}(u_n|I).
	\end{align*}
\end{prop}

\begin{proof}
	The shortest closed geodesics on $S^2$ equipped with the standard metric are of length $\pi$ (the great circles). We choose $p=0$ in  and define 
\begin{align*}
u_\sigma(t)=\frac{\sigma}{(1-2\sigma^2)^{\frac{1}{2}}}\left(\cos\left((1-2\sigma^2)^\frac{1}{2}\frac{t}{\sigma}\right),\sin\left((1-2\sigma^2)^\frac{1}{2}\frac{t}{\sigma}\right),\frac{(1-2\sigma^2)^{\frac{1}{2}}}{\sigma}\sqrt{1-\frac{\sigma^2}{1-2\sigma^2}}\right)
\end{align*}
then $|\dot{u}_\sigma|=1$, and on $S^2$, 
\begin{align*}
D_t\dot{u}_\sigma(t)=\ddot{u}_\sigma(t)=-\frac{(1-2\sigma^2)^\frac{1}{2}}{\sigma}u_\sigma(t)
\end{align*}
so 
\begin{align*}
k(u_\sigma(t))=-\frac{(1-2\sigma^2)^{\frac{1}{2}}}{\sigma}
\end{align*}
and $u_\sigma$ is a critical point of $E_\sigma$ for all $\sigma>0$. And for all $\ens{\sigma_n}_{n\in\N}$ converging to $0$,
\begin{align*}
u_{\sigma_n}\overset{\mathrm{L}^\infty}{\conv} (0,0,1)
\end{align*}
and as $C(\sigma)=\dfrac{1}{2}(1-2\sigma^2)^\frac{1}{2}$, $\K(0)=\dfrac{\pi}{2}$, we have
\begin{align*}
E_\sigma(u_\sigma)=2L(\sigma)(1-\sigma^2)=4\pi\sigma m(\sigma)\frac{1-\sigma^2}{(1-2\sigma^2)^\frac{1}{2}}
\end{align*}
where $m(\sigma)$ is an arbitrary integer. So if we choose
\begin{align*}
\sigma_n=\frac{1}{4n},\quad m(\sigma_n)=n
\end{align*}
then writing $u_n=u_{\sigma_n}$,
\begin{align*}
E_{\sigma_n}(u_{n})\conv \pi=\beta(0),\quad L(u_{n})\conv \frac{\pi}{2}= \frac{\beta(0)}{2}
\end{align*}
while
\begin{align*}
u_{n}\overset{\mathrm{L}^\infty}{\conv} u\equiv(0,0,1)
\end{align*}
and according to Riemann-Lebesgue lemma, $\ens{\dot{u}_n}_{n\in\N}$ converges weakly in $\mathrm{L}^2$ to $\dot{u}$, and if we consider $\ens{\dot{u}_n}_{n\in\N}$ as a sequence of functions on $\R$ (by periodicity), for all $t\in\R/\Q$, $\ens{\dot{u}_n(t)}_{n\in\N}$ has no limit point (as for all $\alpha\in\R/\Q$, $\ens{\cos(n\alpha)}_{n\in\N}$ and $\ens{\sin(n\alpha)}_{n\in\N}$ are dense in $[-1,1]
$). So finally, we have
\begin{align*}
u_n\conv u\quad\text{weakly in}\; \mathrm{W}^{1,2}(S^1,S^2)
\end{align*}
and for all open interval $I\subset S^1$,
\begin{align*}
0=\mathfrak{L}(u(I))<\liminf_{n\rightarrow\infty}\mathfrak{L}(u_n(I))=\frac{\pi}{2}\leb^1(I).
\end{align*}
which concludes the proof of the proposition.
\end{proof}

\subsubsection{Surfaces of Constant Gauss Curvature}

	\begin{prop}
		Let $(M^2,g)$ a compact Riemannian surface of constant Gauss curvature, and $\beta(0)$ is the length of the shortest closed geodesic in $(M^2,g)$. For all $1\leq 2\epsilon<2$, there exists a sequence of positive numbers $\ens{\sigma_n}_{n\in\N}$ converging to $0$, and a sequence of critical points $\ens{u_n}_{n\in\N}$ associated to $\ens{E_{\sigma_n}}_{n\in\N}$ such that
		\begin{align*}
		E_{\sigma_n}(u_n)\conv \beta(0),\quad \mathfrak{L}(u_n)\conv \epsilon \beta(0)
		\end{align*}
		and
		\begin{align*}
		u_n\conv[\mathrm{L}^\infty] u\quad\text{strongly},\quad  u_n\xrightarrow[n\rightarrow\infty]{\mathrm{W}^{1,2}} u\quad\text{weakly},\quad\text{and}\quad \dot{u}_n\centernot{\xrightarrow[n\rightarrow\infty]{}} \dot{u} \quad \text{a.e.}
		\end{align*}
		and
		\begin{align*}
		\mathfrak{L}(u)<\liminf_{n\rightarrow\infty}\mathfrak{L}(u_n).
		\end{align*}
		\end{prop}
\begin{proof} We fix $0\leq p<1$ is fixed, and recall that $K_M\in\R$ is the Gauss curvature. We consider a sequence
$\ens{u_n}_{n\in\N}$ of critical points of $\ens{E_{\sigma_n}}_{n\in\N}$ given by \eqref{formulecounter}, where we chose $\ens{\sigma_n}_{n\in\N}$ and $\ens{m(\sigma_n)}_{n\in\N}$ such that
\begin{align*}
E_{\sigma_n}(u_n)\conv \beta(0)
\end{align*}
The sequence $\ens{u_n}_{n\in\N}$ is bounded in $\mathrm{W}^{1,2}(S^1,M)$, so we can extract a subsequence (still denotes $\ens{u_n}_{n\in\N}$) strongly converging in $\mathrm{L}^{\infty}(S^1,M)$, and weakly converging in $\mathring{W}^{1,2}(S^1,M)$\hspace{0.2em} to a function $u\in\mathrm{W}^{1,2}(S^1,M)$. 

For all interval $I\subset [0,L]$ such that $I\subset [0,L_n]$ for $n$ large enough, we have
\begin{align*}
\mathfrak{L}(u|_{I})^2=\left(\int_{I}|\dot{u}|d\leb^1\right)^2\leq |I|\int_{I}|\dot{u}|^2d\leb^1\leq |I|\liminf_{n\rightarrow\infty} \int_{I}|\dot{u}_n|^2d\leb^1
\end{align*}
and $\ens{u_n}_{n\in\N}$ is in arc-length parametrization, so 
\[
\mathfrak{L}(u_n)^2=|I|\int_{I}|\dot{u}_n|^2d\leb^1
\]
and
\begin{align}
\mathfrak{L}(u|_{I})\leq\liminf_{n\rightarrow\infty}\mathfrak{L}(u_n|_I).
\end{align}
Furthermore, we have
\begin{align}\label{longueurconvergepas}
\mathfrak{L}(u)<\liminf_{n\rightarrow\infty}\mathfrak{L}(u_n).
\end{align}
We prove this assertion by contradiction. If we have the equality in \ref{longueurconvergepas}, then $\ens{\dot{u}_n}_{n\in\N}$ converges almost everywhere to $\dot{u}$, and in particular,
\begin{align*}
L_n=\mathfrak{L}(u_n)\conv \mathfrak{L}(u)=L,
\end{align*}
and $u\in\im$, as we can pass to the limit in the arlength expression $|\dot{u}_n|=1$.
Thanks of the proof of theorem, for all $v\in\mathrm{W}^{2,2}_{u}$, if $v_n=P(u_n)v$ we have
\begin{align}\label{geodegenere}
\int_{0}^{L_n}\scal{\dot{u}_n}{D_tv_n}=-3\int_{0}^{L_n}\sigma_n^2\kappa^2(u_n)\scal{\dot{u}_n}{D_tv_n}+o(1)
\end{align}
and $\ens{D_tv_n}_{n\in\N}$ is bounded in $\mathrm{L}^{\infty}$, so $\ens{\scal{\dot{u}_n}{D_tv_n}}_{n\in\N}$ is bounded in $\mathring{L}^{\infty}$ and is a sequence of continuous functions, while $\ens{\sigma_n^2\kappa^2(u_n)}_{n\in\N}$ converges weakly in $\mathring{L}^{2}$ to $1$. Indeed, 
\begin{align*}
\sigma_n^2\kappa^2(u_n)=\frac{2(1-2\sigma_n^2K_M)}{2-p^2}\left(1-\frac{p^2}{2}\right)+\frac{p^2(1-2\sigma_n^2K_M)}{2-p^2}\cn_p\left(\left(\frac{2(1-2\sigma_n^2K_M)}{2-p^2}\right)^{\frac{1}{2}}\frac{t}{\sigma}\right)
\end{align*}
and the last term converges weakly in $\mathrm{L}^2$ to $0$ according to Riemann-Lebesgue theorem. So we can pass to the limit in \eqref{geodegenere} to find that
\begin{align*}
\int_{0}^{L}\scal{\dot{u}}{D_tv}=-3\int_{0}^{L}\scal{\dot{u}}{D_tv}d\leb^1
\end{align*}
so $u$ is a closed geodesic. As we have chosen $\ens{\sigma_n}_{n\in\N}$, and $\ens{m(\sigma_n})$ such that
\begin{align*}
E_{\sigma_n}(u_n)\conv \beta(0),
\end{align*}
then 
\begin{align*}
L=\epsilon(p)\beta(0)=\left(1+\frac{4}{(2-p^2)\K(p)}\int_{0}^{\frac{\pi}{2}}\sqrt{1-p^2\sin^2(t)}dt\right)^{-1}\beta(0)<\beta(0)
\end{align*}
so $u$ is a non-trivial closed geodesic of length strictly inferior that the length of the shortest closed geodesic, which yields the desired contradiction. Finally as $0\leq p<1$, and
\begin{align*}
	\epsilon(p)\xrightarrow[p\rightarrow 0]{}\frac{1}{2},\quad \epsilon(p)\xrightarrow[p\rightarrow 1]{}1,
\end{align*}
this completes the proof of the proposition.
\end{proof}

\subsubsection{General Surfaces}

In the case of a general surface, we get
\begin{align}\label{Kgen}
k_\sigma(t)=\sigma^2\left(2\ddot{k}_{\sigma}(t)+k_{\sigma}^3(t)+2K\left(\dot{u}(t),\nu(t)\right)k_\sigma(t))\right)
\end{align}
where $K(\dot{u}(t),\nu(t))$ is the sectional curvature of the $2$-plan $\dot{u}(t)\wedge \nu(t)$. Let $K_M^{+}$ (resp. $K_M^{-}$) the maximum (resp. minimum) of sectional curvature of $(M,g)$. If $\mathrm{s}$ is the sign function
\begin{align*}
&\ddot{k}_\sigma(t)+\frac{1}{2}k_\sigma^3(t)+\left(K_M^{\mathrm{s}(k_\sigma(t))}-\frac{1}{2\sigma^2}\right)k_\sigma(t)\geq 0\\
&\ddot{k}_\sigma(t)+\frac{1}{2}k_\sigma^3(t)+\left(K_M^{-\mathrm{s}(k_\sigma(t))}-\frac{1}{2\sigma^2}\right)k_\sigma(t)\leq 0
\end{align*}
If we write $\displaystyle
C(\sigma,K)=\left(\frac{1-2\sigma^2K}{2(2-p(\sigma)^2)}\right)
$, an elementary application of comparison principle implies that there exists a solution $k_\sigma$ of \eqref{Kgen} such that
\begin{equation}\label{comparaison}
\frac{2}{\sigma}C(\sigma,K_M^+)\dn\left(C(\sigma,K_M^+)\frac{t}{\sigma},p(\sigma)\right)\leq k_\sigma(t)\leq \frac{2}{\sigma}C(\sigma,K_M^-)\dn\left(C(\sigma,K_M^-)\frac{t}{\sigma},p(\sigma)\right)
\end{equation}

Furthermore, thanks of a result of Joel Langer and David A. Singer (\cite{LZ1}), we can extend this procedure to get counter-examples in any dimension $m\geq 2$.
Finally inequality \eqref{comparaison} permits to extend the result of the general result of the previous subsection.
	\begin{prop}
		If $(M^2,g)$ is a Riemannian surface, for all $\beta>0$, for all $1<2\epsilon<2$, there exists a sequence of positive numbers $\ens{\sigma_n}_{n\in\N}$ converging to $0$, and a sequence of critical points $\ens{u_n}_{n\in\N}$ associated to $\ens{E_{\sigma_n}}_{n\in\N}$ such that
		\begin{align*}
		E_{\sigma_n}(u_n)\conv \beta,\quad \mathfrak{L}(u_n)\conv \epsilon \beta
		\end{align*}
		and
		\begin{align*}
		u_n\conv[\mathrm{L}^\infty] u\quad\text{strongly},\quad \text{and}\quad u_n\xrightarrow[n\rightarrow\infty]{\mathrm{W}^{1,2}} u\quad\text{weakly},\quad \dot{u}_n\centernot{\xrightarrow[n\rightarrow\infty]{}} \dot{u} \quad \text{a.e.}
		\end{align*}
		and
		\begin{align*}
		\mathfrak{L}(u)<\liminf_{n\rightarrow\infty}\mathfrak{L}(u_n).
		\end{align*}
	\end{prop}

\subsection{Counter-examples in Dimension $2$}

Thanks of an article of Pinkall (see \cite{pinkall}), if $u$ is a critical point of $E_\sigma$, then thanks of the Hopf fibration, we can create an Hopf torus which is a critical point of the Willmore energy. We will take slightly different conventions than the article of Pinkall. Let
 $p:S^3\rightarrow S^2$ the map defined by 
\[
p(w,z)=(|w|^2-|z|^2,2w\bar{z})
\] for all $(w,z)\in S^3$, where \[S^3=\C^2\cap\ens{(w,z): |w|^2+|z|^2=1}
\]
We recall that $p$ is surjective, and we see that it is invariant by the action of $S^1$ by rotation.
It will be convenient for computations to use quaternions for writing Hopf fibration. Le $q\mapsto \tilde{q}$ is the quaternionic  automorphism such that $q$ leaves $1$, $j$ and $k$ unchanged, and which sends $i$ to $-i$. It is easy to verify that the Hopf fibration is given by
\begin{align*}
p(q)=\widetilde{q}q
\end{align*}
if we identify $q=(w,z)\in S^3$. 

Let $\gamma$ a curve $\gamma:[0,L]\rightarrow S^2$ a closed curve of length $L>0$. Let $\Gamma$ a lift of $\gamma$ by the fibration $p$, \textit{i.e.} a curve $\Gamma :[0,L]\rightarrow S^3$ such that $p\circ \Gamma=\gamma$. We now parametrize $\Gamma$ by arclength, and we defined the Hopf torus of $\gamma$ as 
\begin{align*}
\Gamma(t,\theta)=e^{i\theta}\Gamma(t)
\end{align*}
and assume that $t\mapsto \Gamma'(t)$ is orthogonal to $\partial_\theta\Gamma$. This implies that $\Gamma'$ is proportional to $\Gamma$, there exist a smooth function $\lambda$ on $[0,L]$ such that
\begin{align*}
\Gamma'(t)=\lambda(t)\Gamma(t)
\end{align*}
as $\lambda$ is orthogonal to $e^{i\theta}$ for all $\theta\in S^1$, so $\lambda\in\mathrm{Span}(j,k)$. 
To produce the counter-example, we now proceed with the derivation of the mean curvature of the Hopf torus $\Gamma$.

We have
\begin{align*}
\dot{\gamma}=2\widetilde{\Gamma}\lambda\Gamma.
\end{align*}
so $|\dot{\gamma}|=2$. We should be now careful that $\gamma:[0,\frac{L}{2}]\rightarrow S^2$, and $\Gamma:[0,\frac{L}{2}]\times S^1\rightarrow S^3$. We can easily check that $\mathfrak{n}:[0,\frac{L}{2}]\rightarrow S^3$ is a unit normal vector field to the surface $\Gamma$, if 
\begin{align*}
\mathfrak{n}(t,\theta)=ie^{i\theta}\lambda(t)\Gamma(t).
\end{align*}
If we define the function $\kappa$ by the formula 
\begin{align*}
\lambda'=2i\kappa\lambda,
\end{align*}
then 
\begin{align*}
\left\{
\begin{aligned}
\partial_t\mathfrak{n}(t,\theta)&=-2\kappa(t)\partial_t\Gamma(t,\theta)-\partial_\theta\Gamma(t,\theta)\\
\partial_\theta\mathfrak{n}(t,\theta)&=-\partial_t\Gamma(t,\theta)
\end{aligned}\right.
\end{align*}
and $\kappa$ is also the curvature of the curve $\gamma$. Indeed, 
\begin{align*}
\ddot{\gamma}=2\widetilde{\Gamma}\lambda'\Gamma-4\gamma
\end{align*}
so
\begin{align*}
\nabla_{\frac{\dot{\gamma}}{|\dot{\gamma}|}}\frac{\dot{\gamma}}{|\dot{\gamma}|}=\frac{1}{4}4\kappa\widetilde{\Gamma}i\lambda\Gamma=\kappa \nu
\end{align*}
if $\nu$ is the normal of $\gamma$. If we had taken in the beginning a curve parametrized by arclength, then, if  we write $\gamma_0$ the new curve in the arclength parametrization of $\Gamma$,
\begin{align*}
\gamma_0(t)=\gamma(2t),
\end{align*}
and if we now write the curvature with the original curve, we get
\begin{equation}
\left\{
\begin{aligned}
\partial_t\mathfrak{n}(t,\theta)&=-2\kappa(t)\partial_t\Gamma(t,\theta)-\partial_\theta\Gamma(t,\theta)\\
\partial_\theta\mathfrak{n}(t,\theta)&=-\partial_t\Gamma(t,\theta)
\end{aligned}\right.
\end{equation}
The mean curvature is defined as
\begin{align*}
H(t,\theta)=\frac{1}{2}\mathrm{Tr}\,d\mathfrak{n}(t,\theta)
\end{align*}
and the Gaussian curvature by
\begin{align*}
K(t,\theta)=\det d\mathfrak{n}(t,\theta).
\end{align*}
With the new convention about $\kappa$, we have
\begin{align*}
H(t,\theta)=\kappa(2t),\quad K(t,\theta)=-1.
\end{align*}
We now define the Willmore $\sigma$-energy, by
\begin{align*}
W_{\sigma}(\ffg)=\int_{\Sigma}(1+\sigma^2|H_{\ffg}|^2)d\vg
\end{align*}
is $H_{\ffg}$ is a the average of the principal curvature of an immersion $\ffg$ from a Riemannian surface $\Sigma$ in $S^3$.
Then $\ffg$ is a critical point of $\ffg$ if and only if
\begin{align}\label{willmore}
2H=\sigma^2(\Delta_g H+2H(H^2-2K))
\end{align}
if $\Delta_g$ is the Laplace operator for the metric $g$ induced by $\ffg$ on $\Sigma$ by the metric of $S^3$, and $K$ is the Gauss curvature. 
As $|\partial_t\Gamma|=|\partial_\theta\Gamma|=1$, and $\partial_t\Gamma$ is orthogonal to $\partial_\theta\Gamma$, we have
\begin{align*}
\Delta_gH(t,\theta)=\frac{d^2}{dt^2}\kappa(2t)=4\ddot{\kappa}(2t)
\end{align*}
so \eqref{willmore} is equivalent to
\begin{align*}
2\kappa(2t)=\sigma^2(4\ddot{\kappa}(2t)+2\kappa(2t)^3+4\kappa(2t))
\end{align*}
which is equivalent to
\begin{align*}
\kappa=\sigma^2(\ddot{\kappa}+2\kappa^3+2\kappa).
\end{align*}
This last expression is nothing else than \eqref{presque2}, so $\Gamma$ is a critical point of $W_\sigma$ if and only if $\gamma$ is a critical  point of $E_\sigma$. And
\begin{align*}
W_\sigma(\Gamma)=\int_{[0,\frac{L}{2}]}(1+\sigma^2\kappa(2t))dt\,d\theta=\pi E_\sigma(\gamma).
\end{align*}
Furthermore, the second fundamental form $|\mathbb{I}|^2$ is equal to $2+4\kappa^2$, so
\begin{align*}
A_{\sigma}(\Gamma)&=\int_{\Gamma}(1+\sigma^2|\mathbb{I}_\Gamma|^2)d\mathrm{vol}_g=\int_{[0,\frac{L}{2}]\times S^1}(1+2\sigma^2+4\sigma^2\kappa^2(2t))dtd\theta\\&=(1+2\sigma^2)\pi E_{\sigma'}(\gamma),\quad \sigma'=\frac{2\sigma}{\sqrt{1+2\sigma^2}}.
\end{align*}
and as we can show that if $|\mathbb{I}_\Gamma|^2$ depends only of $H$, $\Gamma$ is a critical point of $A_{\sigma}$ if and only if it is a critical point of $W_{\sigma'}$, every $1$-dimension elliptic Jacobi function constructed in the preceding section raises to a critical point of $A_\sigma$.

\begin{prop}
	For all $\beta>0$, there exists a sequence $\ens{\sigma_n}_{n\in\N}$ of positive real numbers converging to $0$, a sequence of flat torii $\ens{T_n^2}_{n\in\N}$ converging to a torus $T^2$, and a sequence $\{\ffg_n:T_n^2\rightarrow S^3\}_{n\in\N}$ of conformal immersions which are critical points associated to $\ens{A_{\sigma_n}}_{n\in\N}$, such that
	\begin{align*}
		\lim_{n\rightarrow\infty}A_{\sigma_n}(\ffg_n)=\beta,\quad \lim_{n\rightarrow\infty}\h^2(\ffg_n(T^2_n))= \frac{\beta}{2},
	\end{align*}
	and $\{\ffg_n\}_{n\in\N}$ weakly converges to a limiting map $\ffg\in \mathrm{W}^{1,2}(T^2,S^3)$, but $\{\ffg_n\}_{n\in\N}$ nowhere strongly converges;  for all open subset $U\subset T^2$
	\begin{align*}
		\h^2(\ffg(T^2\cap U))<\liminf_{n\rightarrow\infty}\h^2(\ffg_n(T_n^2\cap U)).
	\end{align*}
\end{prop}

\begin{proof}
	The proof is now an easy consequence of \eqref{longueurconvergepas}, as we can lift the degenerate family of critical $\ens{u_n}_{n\in\N}$ constructed in the preceding subsection in a family of immersions $\{\ffg_n\}_{n\in\N}$ which are critical points of $\ens{A_{\sigma_n}}$\footnote{Changing the $\sigma_n$ of the $1$-dimensional counter-example into $\sigma_n'$.}, where for all $n\in\N$, 
	\begin{align*}
	T_n=[0,a_n]\times S^1=\R^2/\left(a_n+2\pi i\right)\Z^2\quad \left(a_n=\frac{L_n}{2}\right)
	\end{align*}
	As $L_n\conv L$, and the lifted curves are conformal immersions $\{\ffg_n\}_{n\in\N}$ such that
	\begin{align*}
		|\partial_t\ffg_n|=|\partial_\theta\ffg_n|=1,
	\end{align*}
	this sequence is bounded in $\mathrm{W}^{1,2}(S^1,M)$, so converges weakly to an element $\ffg\in\mathrm{W}^{1,2}(S^1,M)$, and thanks of proposition \ref{countersphere}, for all open subset $U\subset T^2$
\begin{align*}
\h^2(\ffg(T^2\cap U))<\liminf_{n\rightarrow\infty}\h^2(\ffg_n(T_n^2\cap U)),
\end{align*}
which concludes the proof.
\end{proof}

\newpage

\nocite*{}
\bibliographystyle{alpha}
\bibliography{bibliography2}

\end{document}